\documentclass[twocolumn,10pt]{article}
\usepackage[top=.75in, bottom=.75in, left=.65in, right=.65in]{geometry}
\usepackage{amsmath, amsthm, amssymb}
\usepackage{graphicx}
\usepackage{epstopdf}
\usepackage{overpic}
\usepackage{cancel}
\usepackage{rotating}
\usepackage{url}
\usepackage{caption}
\usepackage{color}
\usepackage{rotating}
\usepackage{multirow}
\usepackage{wrapfig}

\usepackage{setspace}
\usepackage{palatino} 
\usepackage[bottom,flushmargin,hang,multiple]{footmisc}
\usepackage{lipsum}
\newcommand\blfootnote[1]{%
  \begingroup
  \renewcommand\thefootnote{}\footnote{#1}%
  \addtocounter{footnote}{-1}%
  \endgroup
}

\DeclareMathOperator*{\argmin}{arg\rm{}min}

\newcommand{\bs}{\mathbf{s}}
\newcommand{\bg}{\mathbf{g}}
\newcommand{\bhs}{\mathbf{\hat{s}}}

\newcommand{\bx}{\mathbf{x}}

\newcommand{\by}{\mathbf{y}}

\newcommand{\bA}{\mathbf{A}}
\newcommand{\bC}{\mathbf{C}}

\newcommand{\bS}{\mathbf{S}}
\newcommand{\bP}{\mathbf{P}}

\newcommand{\bU}{\mathbf{U}}
\newcommand{\bV}{\mathbf{V}}
\newcommand{\bW}{\mathbf{W}}
\newcommand{\bX}{\mathbf{X}}
\newcommand{\bY}{\mathbf{Y}}

\newcommand{\bPsi}{\boldsymbol{\Psi}}
\newcommand{\bphi}{\boldsymbol{\phi}}
\newcommand{\bPhi}{\boldsymbol{\Phi}}
\newcommand{\bSigma}{\boldsymbol{\Sigma}}

\newcommand{\bLambda}{\boldsymbol{\Lambda}}

\definecolor{blue}{rgb}{0,0,1}
\definecolor{darkgreen}{rgb}{0,0.5,0}
\definecolor{red}{rgb}{1,0,0}
\definecolor{teal}{rgb}{0,0.5,0.7}

\newtheorem{definition}{Definition}
\newtheorem{remark}{Remark}
\newtheorem{theorem}{Theorem}
\newtheorem{algorithm}{Algorithm}
\newtheorem{lemma}{Lemma}
\newtheorem{corollary}{Corollary}
\newtheorem{assumption}{Assumption}

\graphicspath{{Figures/}}

\title{\huge{Compressive sampling and dynamic mode decomposition}}
\author{Steven L. Brunton$^{1*}$, Joshua L. Proctor$^2$, J. Nathan Kutz$^1$\\
\small{$^1$ Department of Applied Mathematics, University of Washington, Seattle, WA 98195, United States}\\
\small{$^2$Institute for Disease Modeling, Intellectual Ventures Laboratory, Bellevue, WA 98004, United States}\\ ~ \\
}
\date{}
\begin{document}
\maketitle
\blfootnote{$^*$ Corresponding author. Tel.: +1 609 921 6415.\\ {\indent\emph{E-mail address:} sbrunton@uw.edu (S.L. Brunton).}}
\vspace{-.4in}
\begin{abstract}
This work develops compressive sampling strategies for computing the dynamic mode decomposition (DMD) from heavily subsampled or output-projected data.  
The resulting DMD eigenvalues are equal to DMD eigenvalues from the full-state data.  
It is then possible to reconstruct full-state DMD eigenvectors using $\ell_1$-minimization or greedy algorithms.  
If full-state snapshots are available, it may be computationally beneficial to compress the data, compute a compressed DMD, and then reconstruct full-state modes by applying the projected DMD transforms to full-state snapshots.  

These results rely on a number of theoretical advances.  First, we establish connections between the full-state and projected DMD.  Next, we demonstrate the invariance of the DMD algorithm to left and right unitary transformations.  
When data and modes are sparse in some transform basis, we show a similar invariance of DMD to measurement matrices that satisfy the so-called restricted isometry principle from compressive sampling.  
We demonstrate the success of this architecture on two model systems.  
In the first example, we construct a spatial signal from a sparse vector of Fourier coefficients with a linear dynamical system driving the coefficients.  
In the second example, we consider the double gyre flow field, which is a model for chaotic mixing in the ocean.\\

\end{abstract}

\section{Introduction}
Dynamic mode decomposition (DMD) is a powerful new technique introduced in the fluid dynamics community to isolate spatially coherent modes that oscillate at a fixed frequency~\cite{Rowley:2009,schmid:2010}. 
DMD differs from other dimensionality reduction techniques such as the proper orthogonal decomposition (POD)~\cite{Lumley:1970,Berkooz:1993,HLBR_turb,noack:03cyl}, where modes are selected to minimize the $\ell_2$-projection error of the data onto modes. 
DMD not only provides modes, but also a linear model for how the modes evolve in time.  

The DMD is a data-driven and equation-free method that applies equally well to data from experiments or simulations.  An underlying principle is that even if the data is high-dimensional, it may be described by a low-dimensional attractor subspace defined by a few coherent structures.  When the data is generated by a nonlinear dynamical system, then the DMD modes are closely related to eigenvectors of the infinite-dimensional Koopman operator~\cite{Koopman:1931,Rowley:2009,Mezic:2013}.  In Ref.~\cite{Tu:2014a}, DMD has been shown to be equivalent to linear inverse modeling (LIM)~\cite{penlandMWR89,penlandJClimate93} from climate science, under certain conditions, and it also has deep connections to the eigensystem realization algorithm (ERA)~\cite{kalman:1965,ERA:1985,ERA:2009}.
  
 DMD has been used to study various fluid experiments~\cite{Schmid:2011,Schmid:2012}, shock turbulent boundary layer interaction~\cite{Grilli:2012}, the cylinder wake~\cite{Bagheri:2013}, and foreground/background separation in videos~\cite{Grosek:2013}.  In the context of fluid dynamics, DMD typically relies on time-resolved, full-state measurements of a high-dimensional fluid vector field.  For complex, turbulent flows, it may be prohibitive to collect data across all spatial and temporal scales required for this analysis.  

The present work leverages tools from compressive sampling~\cite{Donoho:2006,Candes:2006a,Candes:2006b,Candes:2006c,Baraniuk:2007,Baraniuk:2009} to facilitate the collection of considerably fewer measurements, resulting in the same dynamic mode decomposition, as illustrated in Figure~\ref{fig:schematic}.  This reduction in the number of measurements may have a broad impact in situations where data acquisition is expensive and/or prohibitive.  In particular, we envision these tools being used in particle image velocimetry (PIV) to reduce the data transfer requirements for each snapshot in time, increasing the maximum temporal sampling rate.  Other applications include ocean and atmospheric monitoring, where individual sensors are expensive.  Even if full-state measurements are available, the proposed method of compressed DMD will be computationally advantageous in many situations where there is low-rank structure in the high-dimensional data.

\subsection{Previous work on sparsity in dynamics}
There are a few examples of prior work utilizing sparsity for the dynamic mode decomposition.  In~\cite{Javanovic:2012}, a sparsity-promoting variant of the dynamic mode decomposition was introduced whereby an $\ell_1$-penalty term on the number of DMD modes balanced the tradeoff between the number of modes and the quality of the DMD representation.  Other algorithms have been developed to obtain only a fixed number of modes, but these have involved global minimization techniques that may not scale with large problems~\cite{Chen:2012}.

In~\cite{Tu:2014b} and~\cite{Glauser:2013}, compressive sampling has been used to design non-time resolved sampling strategies for particle image velocimetry (PIV) of a fluid velocity field; in \cite{Tu:2014b}, this sampling is specifically used to compute DMD.  These experimental methods are based on the fact that temporally sparse signals may be sampled considerably less often than suggested by the Shannon-Nyquist sampling frequency~\cite{Nyquist:1928,Shannon:1948}.  

In~\cite{Shi:2014}, compressive sampling is paired with the theory of linear dynamical systems to obtain higher temporal sampling resolution and accurate reconstruction of video MRI.  Their work is based on prior studies relating compressive sampling, linear dynamical systems, and video MRI~\cite{Sankaranarayanan:2012,Patel:2013}.  Incoherent measurements are used to estimate an underlying snapshot matrix of hidden-Markov states, as well as an embedding from this low-dimensional attractor into the high-dimensional image pixel space.  For the first part, they use system identification based on Hankel matrices and minimal realization theory~\cite{kalman:1965}.  Based on the heavy use of Hankel matrices in ERA, and the established connections between DMD and ERA, it will be interesting to see how the present work connects to~\cite{Shi:2014} in the future.

\subsection{Contribution of this work}
This work deviates from the prior studies combining compressive sampling and DMD, in that we utilize sparsity of the \emph{spatial} coherent structures to reconstruct full-state DMD modes from few measurements.  This method results in full-state DMD modes from spatially projected or subsampled measurements using compressive sampling.  The eigenvalues of the projected DMD are equal to the full-state DMD eigenvalues, so that we obtain the same low-dimensional model to advance mode coefficients.  

These results highlight the ability to perform DMD with significantly less data acquisition when the data and modes are sparse in some basis.  If full-state snapshots are available, it is possible to pre-compress the data, compute a projected DMD, and then reconstruct full-state DMD modes as a linear combination of the original full-state data.  The compressed DMD provides significant computational savings over traditional DMD.  These methods are described in Section~\ref{sec:pathways} as various paths we can take in Figure~\ref{fig:schematic}, depending on initial data.

Our results rely on a number of theoretical advances that may be useful more broadly.  First, we establish connections between the full-state and projected DMD.  We then show that DMD is invariant to left and right unitary transformations of the data.  This implies that the DMD computed in the spatial domain, Fourier domain, or in a POD coordinate system will all be closely related, since these coordinate systems are related by unitary transformations.  We then show that when data and modes are sparse in some basis, we obtain a similar invariance of the DMD when our measurement matrix and sparse basis satisfy the restricted isometry principle.   

The methods in this paper are illustrated on two examples that are relevant to fluid mechanics.  However, we believe that there is broad applicability of compressive sampling in dynamical systems more generally.

\section{Background}\label{sec:background}
Dynamic mode decomposition is a method of modal extraction from full-state snapshot data that results in spatial-temporal coherent structures oscillating with a fixed frequency and damping rate.  This theory has recently been generalized and extended to a larger class of datasets~\cite{Tu:2014a}, and it is discussed in Sec.~\ref{ss:DMD}.

The present work is centered around the use of compressive sampling to compute the dynamic mode decomposition from very few spatial measurements.  In compressive sampling, a high-dimensional signal may be reconstructed from few measurements as long as the signal is sparse in some transform basis.  
We discuss compressive sampling in Sec.~\ref{ss:CS}.  

\subsection{Dynamic mode decomposition (DMD)}
\label{ss:DMD}
The dynamic mode decomposition (DMD) is a new tool in dynamical systems that has been introduced in the fluid dynamics community~\cite{Rowley:2009,schmid:2010}.  The DMD provides the eigenvalues and eigenvectors of the best-fit linear system relating a snapshot matrix and a time-shifted version of the snapshot matrix at some later time.      

Consider the following data snapshot matrices:
\begin{align*}
\bX = \begin{bmatrix} \vline & \vline & & \vline \\
	\bx_0 & \bx_1 & \hspace{-.05in}\cdots  \hspace{-.05in} & \bx_{m-1}\\
	\vline & \vline &  & \vline\end{bmatrix},
\bX'= \begin{bmatrix} \vline & \vline & & \vline \\
	\bx_1 & \bx_2 & \hspace{-.05in} \cdots  \hspace{-.05in}& \bx_{m}\\
	\vline & \vline &  & \vline\end{bmatrix}.
\end{align*}
Here, $\bx_k\in\mathbb{R}^n$ is the $k^{\text{th}}$ snapshot, and typically $n\gg m$.  $\bx$ is often the state of a high-dimensional dynamical system, such as a fluid flow.  We also consider the snapshots to be spaced evenly in time, and we may also imagine that these measurements are discrete-time samples of a continuous-time signal, so that $\bx_k=\bx(k\Delta t)$.  

The dynamic mode decomposition involves the decomposition of the best-fit linear operator $\bA$ relating the matrices above:
\begin{align}
\bX'=\bA\bX.
\label{eq:dmd}
\end{align}
When ambiguous, we may refer to $\bA$ in Eq.~(\ref{eq:dmd}) as $\bA_\bX$.

The exact DMD algorithm proceeds as follows:
\begin{algorithm}
The method of \emph{\underline{exact DMD}} was recently defined~\cite{Tu:2014a}, and it is given by the following procedure:
\begin{enumerate}
\item Collect data $\bX, \bX'$ and compute the singular value decomposition (SVD) of $\bX$:
\begin{align}
\bX = \bU\bSigma\bV^*.
\label{eq:dmdalg1}
\end{align}
\item Compute the least-squares fit $\bA$ that satisfies ${\bX'=\bA\bX}$ and project onto POD/PCA modes $\bU$:
\begin{align}
\tilde{\bA} = \bU^*\bA\bU = \bU^*\bX'\bV\bSigma^{-1}.
\label{eq:dmdalg2}
\end{align}
\item Compute the eigen-decomposition of $\tilde{\bA}$:
\begin{align}
\tilde{\bA}\bW = \bW\bLambda.
\label{eq:dmdalg3}
\end{align}
$\bLambda$ are the DMD eigenvalues.
\item Compute the DMD modes $\bPhi$:
\begin{align}
\bPhi = \bX'\bV\bSigma^{-1}\bW.
\label{eq:dmdalg4}
\end{align}
\end{enumerate}
\end{algorithm}

\begin{remark}
The first three steps in the algorithm above are identical to those in~\cite{schmid:2010}.  However, the last step differs in the computation of DMD modes.  In~\cite{schmid:2010}, the modes are given by $\bPhi = \bU\bW$.  In~\cite{Tu:2014a}, this formula $\bPhi=\bU\bW$ is still used to compute modes corresponding to zero eigenvalues.
\end{remark} 

The data $\bX,\bX'$ may come from a nonlinear system
\begin{align*}
\bx_{k+1} = \mathbf{f}(\bx_k),
\end{align*}
in which case the DMD modes are related to eigenvectors of the infinite-dimensional Koopman operator $\mathcal{K}$ which acts as the pull-back operator on observable functions~\cite{Koopman:1931,MarsdenMTAA,Rowley:2009,Mezic:2013}.  In particular, $\mathcal{K}$ acts on observable functions $\bg$ as:
\begin{align*}
\mathcal{K}\bg(\bx_k)=\bg(\bf{f}(\bx_k))=\bg(\bx_{k+1}).
\end{align*}

The connection between DMD and the Koopman operator justify the application of this method in a variety of contexts.  We may interpret DMD as a model reduction technique if data is acquired from a high-dimensional model, or a method of system identification if the data comes from measurements of an uncharacterized system.  In the latter case, the resulting DMD model is data-driven and may be used in conjunction with equation-free methods~\cite{Kevrekidis:2003}.  The hierarchy of structure in the data is illustrated in Figure~\ref{fig:data}.  Recently, the assumption of evenly spaced snapshots was relaxed, so that the columns of $\bX$ may be sampled at any times, as long as the columns of $\bX'$ are sampled a fixed $\Delta t$ later~\cite{Tu:2014a}.

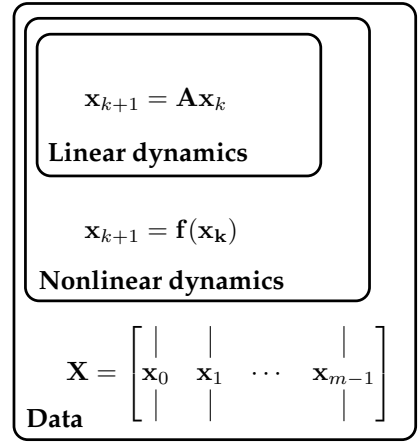
\begin{figure}
\begin{center}
\begin{overpic}[width=0.31\textwidth]{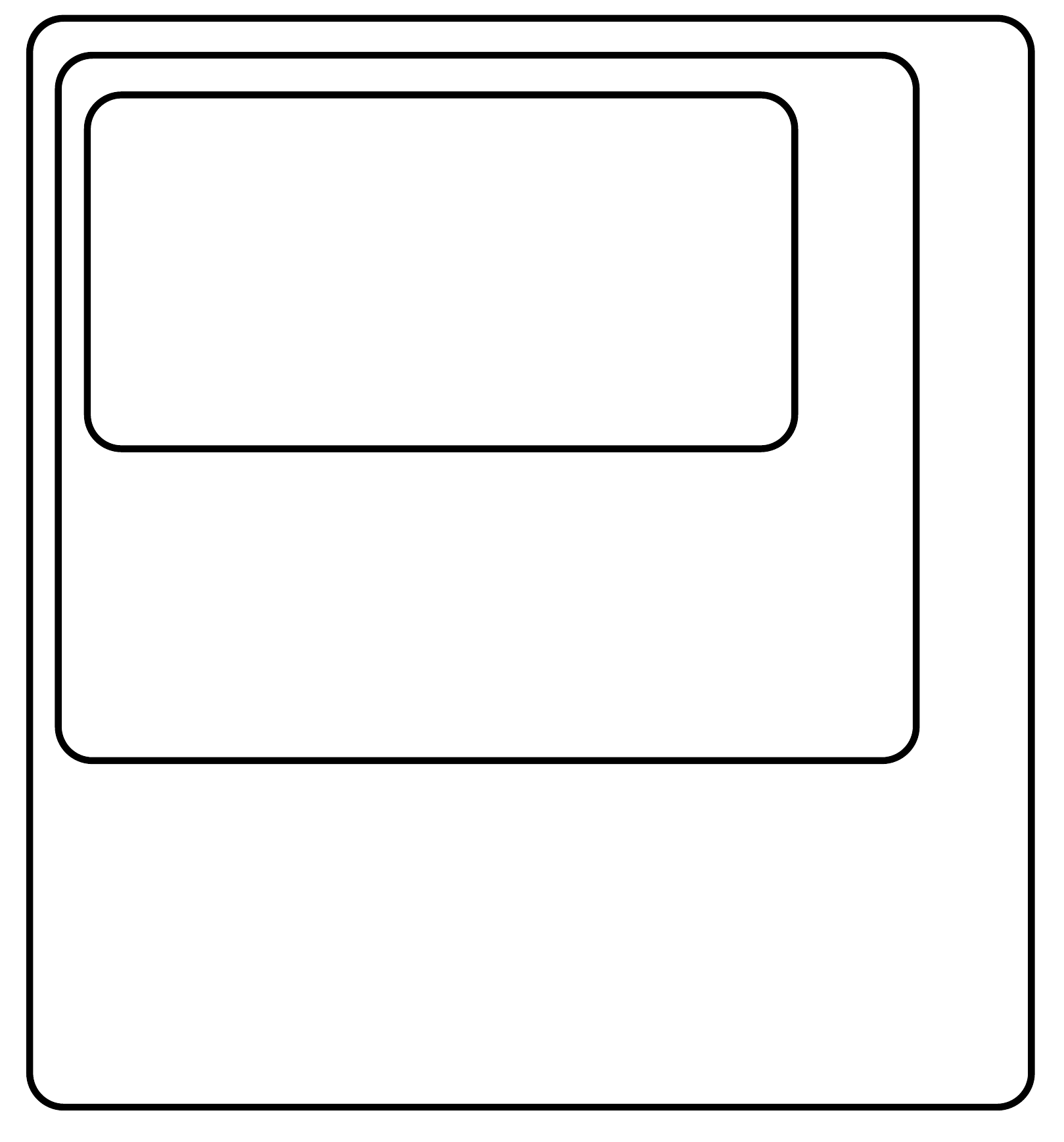}
\put(5.25,5.75){\textbf{Data}}
\put(8,36){\textbf{Nonlinear dynamics}}
\put(10,64){\textbf{Linear dynamics}}
\put(18,76){$\bx_{k+1} = \bA \bx_k$}
\put(18,47){$\bx_{k+1}=\bf{f}(\bx_k)$}
\put(14,16){$\bX = \begin{bmatrix} \vline & \vline & & \vline \\
	\bx_0 & \bx_1 & \cdots & \bx_{m-1}\\
	\vline & \vline &  & \vline\end{bmatrix}$}
\end{overpic}
\caption{Schematic of various assumptions of dynamic structure underlying data $\bX,\bX'$.}\label{fig:data}
\vspace{-.2in}
\end{center}
\end{figure}

\subsection{Compressive sampling}
\label{ss:CS}
Consider a signal $\bx\in\mathbb{R}^n$, which is sparse in some basis given by the columns of $\bPsi$, so that
\begin{align}
\bx = \bPsi\bs,
\label{eq:sparse}
\end{align}
and $\bs$ is a vector containing mostly zeros.  The signal $\bx$ is $K$-sparse if $\bs$ has exactly $K$ nonzero elements.  

Most natural signals are sparse in some basis.  For example, natural images and audio signals are sparse in Fourier or wavelet bases, resulting in a high-degree of \emph{compressibility}.  If we take the Fourier or Wavelet transform of an image, most of the coefficients will be small and can be neglected without resulting in much loss of image quality.  Truncating in Fourier or Wavelet bases is the foundation of JPEG-2000 image compression and MP3 audio compression.  Similarly, many high-dimensional nonlinear PDEs have sparse solutions~\cite{Schaeffer:2013}.

The theory of compressive sampling~\cite{Donoho:2006,Candes:2006a,Candes:2006b,Candes:2006c,Baraniuk:2007,Baraniuk:2009} suggests that instead of measuring the high-dimensional signal $\bx$ and then compressing, it is possible to measure a low-dimensional subsample or random projection of the data and then directly solve for the few non-zero coefficients in the transform basis.  
Consider the measurements $\by\in\mathbb{R}^p$, with $K< p\ll n$:
\begin{align*}
\by = \bC\bx.
\end{align*}
The measurement matrix $\bC$ is often denoted by $\bPhi$ in the compressive sampling literature.  However, $\bPhi$ is already used in the DMD community for DMD modes in Eq.~(\ref{eq:dmdalg4}).  

If $\bx$ is sparse in $\bPsi$, then we would like to solve the underdetermined system of equations
\begin{align}
\by=\bC\bPsi\bs \label{eq:underdet}
\end{align}
for $\bs$ and then reconstruct $\bx$.  Since there are infinitely many solutions to this system of equations, we seek the sparsest solution $\bhs$,
\begin{align}
\bhs=\argmin_{\bs'}\|\bs'\|_0,\text{ such that }\by=\bC\bPsi\bs'.\label{eq:l0min}
\end{align}
However, this amounts to a brute force combinatorial search, which is infeasible for even moderately large problems.  Under certain conditions on the measurement matrix $\bC$, Eq.~(\ref{eq:l0min}) may be relaxed to a convex $\ell_1$-minimization~\cite{Candes:2006c,Donoho:2006}:
\begin{align}
\bhs=\argmin_{\bs'}\|\bs'\|_1,\text{ such that }\by=\bC\bPsi\bs'.\label{eq:l1min}
\end{align}
Specifically, the measurement matrix $\bC$ must be \emph{incoherent} with respect to the sparse basis $\bPsi$, so that rows of $\bC$ are uncorrelated with columns of $\bPsi$.  In this case, the matrix $\bC\bPsi$ satisfies the \emph{restricted isometry principle} (RIP) for sparse vectors $\bs$,
\begin{align*}
(1-\delta_K)\|\bs\|_2^2\leq \|\bC\bPsi \bs\|_2^2\leq (1+\delta_K)\| \bs \|_2^2,
\end{align*}
with restricted isometry constant $\delta_K$.  $\bC\bPsi$ acts as a near isometry on $K$-sparse vectors $\bs$.  The RIP will be an important part of the following analysis combining sparsity and dynamic mode decomposition.  In addition to taking incoherent measurements, we must take on the order of $K\log(n/K)$ measurements to accurately determine the $K$ nonzero elements of the $n$-length vector $\bs$~\cite{Candes:2006,Candes:2006a,Baraniuk:2007}.

Typically a generic basis such as Fourier or wavelets is used to represent the sparse signal $\bs$.  The Fourier transform basis is particularly attractive for engineering purposes since single-pixel measurements are incoherent, exciting broadband frequency content.  If a signal is $K$-sparse in the Fourier domain, we may then reconstruct the full state from $\mathcal{O}(K\log(n/K))$ single-pixel measurements at random spatial locations.  Random pixel sampling is especially beneficial when individual measurements are expensive, for example in ocean and atmospheric sampling, among other applications.  

Another major result of compressive sampling is that Bernouli and Gaussian random measurement matrices $\bC$ will satisfy the RIP with high probability for a generic basis $\bPsi$~\cite{Candes:2006b}.  There is also work describing incoherence with sparse matrices and generalizations to the RIP~\cite{Gilbert:2010}.  
Recent work has shown the advantage of pairing compressive sampling with a data-driven POD/PCA basis, in which the data is optimally sparse~\cite{Kutz:2013,Bright:2013,Glauser:2013,Brunton:2014a}.  The use of a POD/PCA basis results in a more computationally efficient signal reconstruction from fewer measurements.  

In addition to the $\ell_1$ minimization described in Eq.~(\ref{eq:l1min}), there are a host of \emph{greedy algorithms}~\cite{Tropp:2004,Tropp:2010,Needell:2010,Gilbert:2012} that iteratively determine the sparse solution to the underdetermined system in Eq.~(\ref{eq:underdet}).  There has also been significant work on compressed SVD and PCA based on the Johnson-Lindenstrauss (JL) lemma~\cite{JL:1984,Fowler:2009,Qi:2012,Gilbert:2012}.  The JL lemma is closely related to the RIP, and it states when it is possible to embed a set of high-dimensional vectors in a low-dimensional space while preserving the spectral properties.

\section{Compressive DMD}\label{sec:methods}
In this section, we combine ideas from compressive sampling to compute the dynamic mode decomposition from a few spatially incoherent measurements.  In Sec.~\ref{sec:csdmd:1}, we establish basic connections between the DMD on full-state and projected data.  These connections facilitate the two main applied results of this work:
\begin{itemize}
\item [1)] It is possible to compute DMD on projected data and reconstruct full-state DMD modes through compressive sampling.
\item[2)] If full-state measurements are available, it is advantageous to compress the data, compute the projected DMD, and then compute full-state DMD modes by linearly combining snapshots according to the projected DMD transforms.  
\end{itemize}
In both cases, the full-state and projected DMD eigenvalues are equal.  These two approaches are described in Section~\ref{sec:pathways} as various pathways to take in Figure~\ref{fig:schematic}, depending on what the initial data is.

Next, in Sec.~\ref{sec:csdmd:2}, we demonstrate the invariance of the DMD algorithm to left or right unitary transformations of the data.  We then discuss how the condition of unitarity may be relaxed to a transformation satisfying a restricted isometry principle, as long as the data is sparse in a basis that is incoherent with respect to the measurements.  This strengthens the connection to compressive sampling.  Full-state DMD modes are then reconstructed from projected DMD modes using compressive sampling; in particular, we use the method of compressive sampling matching pursuit (CoSaMP)~\cite{Needell:2010}.

\subsection{Projected DMD eigenvalues and eigenvectors}\label{sec:csdmd:1}

It is possible to either collect data $\bX,\bX'$, or output-projected data $\bY,\bY'$, where $\bY=\bC\bX$, ${\bY'=\bC\bX'}$ and $\bC\in\mathbb{R}^{p\times m}$ is the measurement matrix.

\begin{definition}
We refer to $\bX$ and $\bX'$ as the \emph{full-state snapshot matrices} and $\bY$ and $\bY'$ as the \emph{output-projected snapshot matrices}.
\end{definition}

Similar to Eq.~(\ref{eq:dmd}) above, $\bY$ and $\bY'$ are related by
\begin{align}
\bY' = \bA_\bY \bY.
\label{eq:dmd2}
\end{align}

We may also rely on the following three assumptions of sparsity of data and incoherence of measurements.

\begin{assumption}\label{ass1}
The columns of $\bX$ and $\bX'$ are sparse in a transform basis $\bPsi$ so that $\bX = \bPsi\bS$ and $\bX' = \bPsi\bS'$, where the columns of $\bS,\bS'$ are sparse (mostly zeros).
\end{assumption}

\begin{assumption}\label{ass2}
The measurement matrix $\bC$ is \emph{incoherent} with respect to $\bPsi$ so that a restricted isometry principle is satisfied.
\end{assumption}

\begin{remark}
Under Assumptions \ref{ass1} and \ref{ass2}, it is possible to reconstruct $\bx$ from $\by=\bC\bx$, and therefore $\bX,\bX'$ from $\bY,\bY'$ using compressive sampling.  However, this is laborious and inelegant.
\end{remark}

\begin{assumption}\label{ass3}
In addition to sparsity of the columns of $\bX$ and $\bX'$, we may also require each of the columns to be in the same sparse subspace of the basis $\bPsi$.  The POD modes $\bU_\bX$ and the DMD modes $\bPhi_\bX$ are then guaranteed to be in this sparse subspace.  This condition is reasonable when there is a low-dimensional underlying dynamical system in the sparse basis that strongly damps all but a small subspace.
\end{assumption}

\begin{lemma}
The full-state and projected DMD matrices $\bA_\bX$ from Eq.~(\ref{eq:dmd}) and $\bA_\bY$ from Eq.~(\ref{eq:dmd2}) are related as:
\begin{align}
\bC \bA_\bX = \bA_\bY \bC.
\end{align}
\label{lemma:dmd}
\end{lemma}
\begin{proof}
We may substitute $\bY=\bC\bX$ in $\bY'=\bA_\bY\bY$:
\begin{align*}
\bC\bX' = \bA_\bY\bC\bX.
\end{align*}
Taking the right pseudo-inverse of $\bX$ yields
\begin{align*}
&\bC\bX'\bV_\bX\bSigma_\bX^{-1}\bU_\bX^*=\bA_\bY\bC\\
\Longrightarrow\quad&\bC\bA_\bX=\bA_\bY\bC.
\end{align*}
\end{proof}

The above lemma allows us to prove the following theorem, which establishes the central connection between DMD on full and compressed data. 

\begin{theorem}
An eigenvector $\bphi_x$ of $\bA_\bX$ projects to an eigenvector $\bC\bphi_x$ of $\bA_\bY$ with the same eigenvalue $\lambda$.
\label{thm:dmd}
\end{theorem}

\begin{proof}
\begin{align*}
&\bC\bA_\bX\bphi_x = \bA_\bY\bC\bphi_x\\
&\lambda \bC\bphi_x = \bA_\bY\bC\bphi_x
\end{align*}
so $\bC\bphi_x$ is an eigenvector of $\bA_\bY$ with eigenvalue $\lambda$.
\end{proof}

If $\bC$ is chosen poorly, so that $\bphi_x$ is in the null-space of $\bC$, then Theorem~\ref{thm:dmd} applies trivially.  
Theorem~\ref{thm:dmd} \emph{does not} guarantee that every eigenvector $\bphi_y$ of $\bA_\bY$ is the projection of an eigenvector of $\bA_\bX$ through $\bC$.  
However, it is typically the case that the rank $r$ of $\bX$ is small compared with $m$, the number of columns, and $p$, the number of output measurements, so that $r<p,m\ll n$.  In fact, the assumption of low-rank structure is implicit in most dimensionality reduction strategies.  
As long as the columns of $\bX$ are not in the null-space of $\bC$ and the rank of $\bY=\bC\bX$ is also $r$, then the $r$ nontrivial DMD eigenvalues $\bLambda_\bY$ of $\bY,\bY'$ will be be equal to $\bLambda_\bX$.  Similarly, the projected DMD eigenvectors are related to full-state DMD eigenvectors as described in Theorem~\ref{thm:dmd}.

\begin{corollary}
Given Assumptions 1-3, we may reconstruct $\bphi_x$ from the eigenvector ${\bphi_y=\bC\bphi_x=\bC\bPsi\bphi_s}$ from Theorem~\ref{thm:dmd} by compressive sampling.
\end{corollary}

\begin{remark}
Even starting with full-state measurements $\bX,\bX'$, it is beneficial to compress, compute the DMD, and reconstruct the modes according to:
\begin{align}
\tilde{\bPhi}_{\bX} = \bX'\bV_{\bY}\bSigma_{\bY}^{-1}\bW_{\bY}.\label{eq:compressedDMD}
\end{align}
We refer to this as \emph{\underline{compressed DMD}}, as opposed to the \emph{\underline{compressive-sampling DMD}} above.
\end{remark}

\begin{remark}
Lemma~\ref{lemma:dmd} and Theorem~\ref{thm:dmd} apply when $\bX,\bX'$ and $\bY,\bY'$ are \emph{exactly} related by $\bA_\bX$ and $\bA_\bY$ from Eqs.~(\ref{eq:dmd}) and (\ref{eq:dmd2}), as is the case when $\bX,\bX'$ are generated by a linear dynamical system (or more generally, when $\bX'$ is in the column space of $\bX$).  However, often $\bX'\approx \bA_\bX\bX$, where $\bA_\bX = \bX'\bX^\dagger\triangleq \bX'\bV_\bX\bSigma_\bX^{-1}\bU_\bX^*$.  In this case, we require that $(\bC\bX)^\dagger\approx\bX^\dagger\bC^*$, which holds either when $\bC$ satisfies the Johnson-Lindenstrauss theorem, or $\bX=\bPsi\bS$ and $\bC\bPsi$ satisfies the restricted isometry principle, as in the next section.
\end{remark}

\subsection{Invariance of DMD to unitary transformations}\label{sec:csdmd:2}
In this section, we show that the dynamic mode decomposition is invariant to left and right unitary transformations of the data $\bX$ and $\bX'$.  This is the theoretical foundation for the next step, where we relax the condition of a unitary measurement matrix, and instead require that our data is sparse in some basis and the measurements are incoherent with respect to that basis.  

This section relies on the fact that the singular value decomposition of $\bY=\bC\bX$ is related to the singular value decomposition of $\bX=\bU\bSigma \bV^*$ if $\bC$ is unitary:
\begin{align}
\bY = \bC\bU\bSigma \bV^*
\end{align}
So now $\bC\bU$ are the left-singular vectors of $\bY$.  Similarly, if $\bY = \bX\bP^*$, then $\bY = \bU\bSigma\bV^*\bP^*$.  This is discussed in more detail in the Appendix.

\begin{theorem}\label{th1}
The DMD eigenvalues and eigenvectors are invariant to right-transformations $\bP$ of the columns of $\bX$ and $\bX'$ if $\bP$ is unitary.
\end{theorem}

\begin{proof}
Let $\bP$ be a $m\times m$ unitary matrix that acts on the columns of $\bX$ and $\bX'$ as: $\bY=\bX\bP$, and $\bY'=\bX'\bP$.  The four steps of the exact DMD algorithm proceed as:
\begin{enumerate}
\item $\bY=\bU\bSigma \bV^*\bP$
\item $\tilde{\bA}_{\bY}=\bU^*\bY'\bP^*\bV\bSigma^{-1}= \bU^*\bX'\bV\bSigma^{-1}=\tilde{\bA}_{\bX}$
\item $\tilde{\bA}_{\bY}\bW_{\bY}=\bW_{\bY}\bLambda_{\bY}$ same as $\tilde{\bA}_{\bX}\bW_{\bX}=\bW_{\bX}\bLambda_{\bX}$, so $\bW_{\bY}=\bW_{\bX}$ and $\bLambda_{\bY}=\bLambda_{\bX}$.
\item $\bPhi_{\bY} = \bY'\bP^*\bV\bSigma^{-1}\bW_{\bY} = \bPhi_{\bX}$.
\end{enumerate}
Therefore, the DMD eigenvalues and eigenvectors are invariant to right-unitary transformations of the data. 
\end{proof}

\begin{corollary}
The DMD eigenvalues and eigenvectors are invariant to permutations of the columns of $\bX$ and $\bX'$\footnote{This was communicated by Jonathan Tu in a conversation.}.
\end{corollary}

\begin{theorem}\label{thm3}
The DMD eigenvalues are invariant to left-transformations $\bC$ of data $\bX$ if $\bC$ is unitary, and the resulting DMD modes are projected through $\bC$.  
\end{theorem}

\begin{proof}
Again, let $\bY=\bC\bX$ and $\bY'=\bC\bX'$.  Exact DMD proceeds as follows:
\begin{enumerate}
\item $\bY = \bC\bU\bSigma\bV^*$
\item $\tilde{\bA}_{\bY} = \bU^*\bC^*\bY'\bV\bSigma^{-1} = \bU^*\bX'\bV\bSigma^{-1} = \tilde{\bA}_{\bX}$
\item The eigen-decomposition ($\bLambda_{\bY}$,$\bW_{\bY}$) is the same as ($\bLambda_{\bX}$,$\bW_{\bX}$) since $\tilde{\bA}_{\bY}=\tilde{\bA}_{\bX}$.
\item $\bPhi_{\bY} = \bY'\bV\bSigma^{-1}\bW = \bC\bPhi_{\bX}$.
\end{enumerate}
Therefore, the DMD modes $\bPhi_{\bY}$ are the projection of $\bPhi_{\bX}$ through $\bC$:  $\bPhi_\bY=\bC\bPhi_\bX$.
\end{proof}

\begin{corollary}
The DMD computed in the spatial domain is related to the DMD computed in the Fourier domain or in a coordinate system defined by principle components.  Both the discrete Fourier transform $\mathcal{F}$ and the principal component coordinate transformation $\bU$ are unitary transformations, and so Theorem~\ref{thm3} applies with $\bC=\mathcal{F}$ or $\bC=\bU$, respectively.  The DMD eigenvalues will be unchanged, and the DMD eigenvectors will be projected into the new coordinates.
\end{corollary}

Now we may consider data $\bX,\bX'$ that are sparse in some basis $\bPsi$, and relax the condition that $\bC$ is unitary.  Instead, $\bC$ must be incoherent with respect to the sparse basis, so that the product $\bC\bPsi$ satisfies a restricted isometry principle.  This allows us to compute DMD on the projected data $(\bY,\bY')=(\bC\bX,\bC\bX')=(\bC\bPsi\bS,\bC\bPsi\bS')$ and reconstruct full-state DMD modes by compressive sampling on $\bPhi_\bY=\bC\bPhi_\bX=\bC\bPsi\bPhi_\bS$.

\begin{figure}
\begin{center}
\vskip .2in
\begin{overpic}[width=0.45\textwidth]{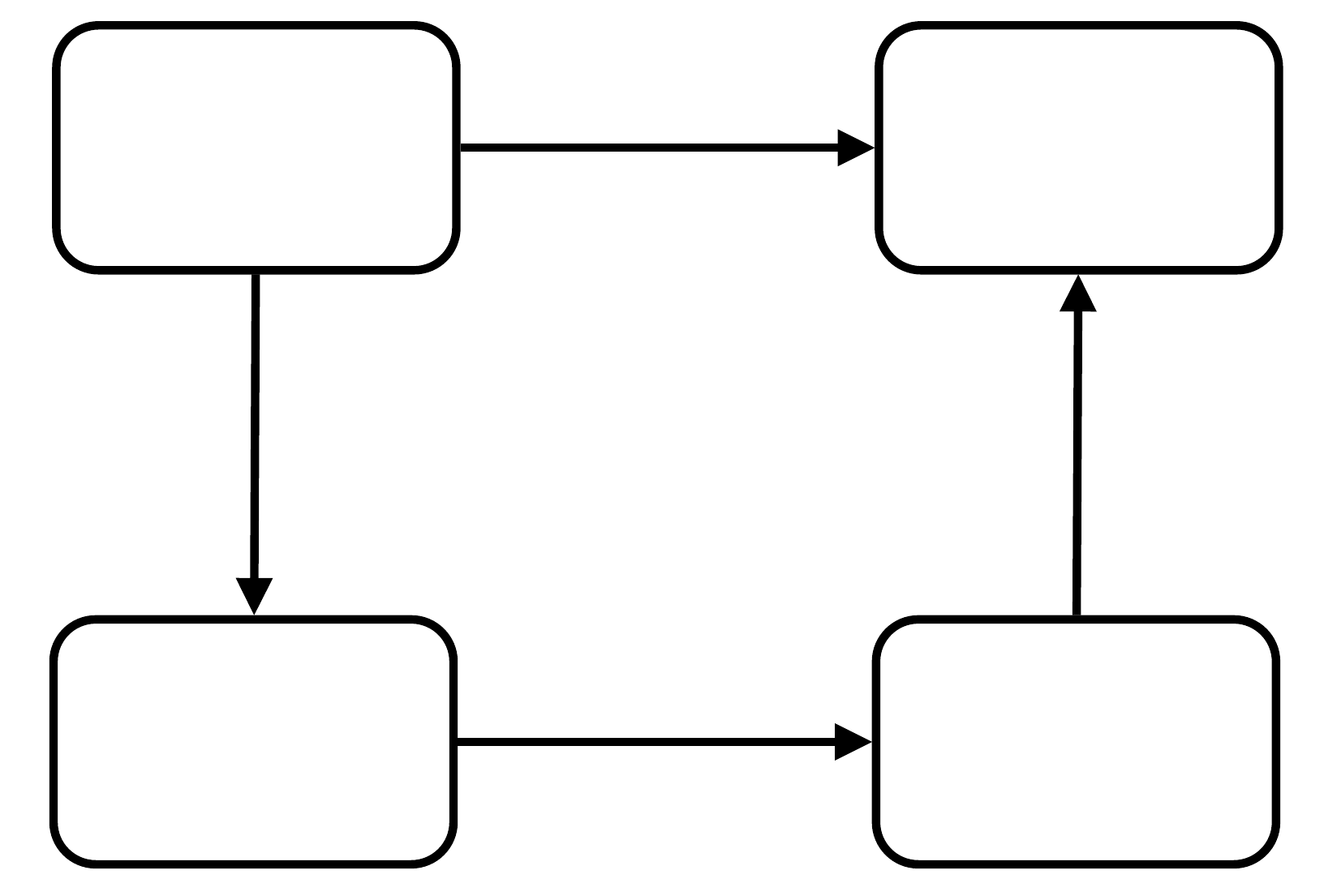}
\put(14,55){$\bX, \bX'$}
\put(74,55){$\boldsymbol{\Lambda_{\bX}},\bPhi_{\bX}$}
\put(14,10){$\bY, \bY'$}
\put(74,10){$\boldsymbol{\Lambda}_{\bY},\bPhi_{\bY}$}
\put(12,33.5){$\bC$}
\put(45,60){DMD}
\put(41.5,18){Projected}
\put(45,13){DMD}
\put(84,24){\begin{sideways}reconstruct\end{sideways}}
\put(13,68.5){\large \bf Data}
\put(73.5,68.5){\large \bf Modes}
\end{overpic}
\caption{Schematic of compressive-sampling and compressed DMD as they relate to data $\bX,\bX'$ and projected data $\bY,\bY'$.  $\bC$ is a projection down to measurements that are incoherent with respect to sparse basis.}\label{fig:schematic}
\vspace{-.2in}
\end{center}
\end{figure}
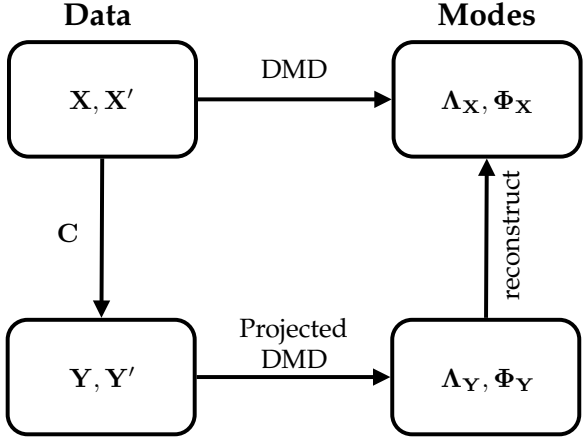
\subsection{Various approaches and algorithms}\label{sec:pathways}
There are a number of algorithms that arise from various paths in Figure~\ref{fig:schematic} depending on what data we have access to.  The primary paths are: Path 1B (compressed DMD) and Path 2B (compressive sampling DMD).\\

\noindent\textbf{\underline{Path 1}:}  We start with full-state data $\bX,\bX'$.

	\begin{itemize}
	\item[] \hskip -.15in{\bf \underline{Option A}.}  Compute DMD to obtain $(\bLambda_{\bX},\bPhi_{\bX})$.
	\item[] \hskip -.15in{\bf \underline{Option B}.}  Compress data first:
		\begin{enumerate}
		\item[(i)]  Compress $\bX,\bX'$ to $\bY,\bY'$.
		\item[(ii)]  Compute DMD to obtain $(\bLambda_{\bY},\bPhi_{\bY})$ and $\bW_\bY$.
		\item[(iii)]  Reconstruct $\tilde\bPhi_{\bX} = \bX'\bV_{\bY}\bSigma_{\bY}^{-1}\bW_{\bY}$.
		\item[(iv)] [alternative] Perform $\ell_1$-minimization on ${\bPhi_{\bY}=\bC\bPsi\bPhi_{\bS}}$ to reconstruct $\bPhi_{\bS}$ and then construct $\bPhi_{\bX}=\bPsi\bPhi_{\bS}$.
		\end{enumerate}
	\end{itemize}

\noindent\textbf{\underline{Path 2}:}  We only have output-projected data $\bY,\bY'$.
	\begin{itemize}
	\item[] \hskip -.15in{\bf \underline{Option A}.}  First reconstruct $\bX,\bX'$ using compressive sampling.
		\begin{itemize}
		\item[(i)] Perform $\ell_1$-minimization on $\bY=\bC\bPsi\bS$ to solve for $\bS$, and hence $\bX$ (same for $\bX'$).
		\item[(ii)] Compute DMD on $\bX$ (or $\bS$).
		\end{itemize}
	\item[] \hskip -.15in{\bf \underline{Option B}.}  Compute projected DMD and only reconstruct $r$ modes using compressive sampling.
		\begin{itemize}
		\item[(i)] Compute DMD on $\bY,\bY'$ to obtain $(\bLambda_{\bY},\bPhi_{\bY})$
		\item[(ii)] Perform $\ell_1$-minimization on $\bPhi_\bY$ to solve ${\bPhi_{\bY}=\bC\bPsi\bPhi_{\bS}}$ for $\bPhi_{\bS}$ (and hence $\bPhi_{\bX}$).
		\end{itemize}
	\end{itemize}
As a general rule, if you have the full data $\bX,\bX'$, Path 1B (i)-(iii) is the most efficient option; we refer to this as \emph{\underline{compressed DMD}}.  If only the output projected data $\bY,\bY'$ is available, Path 2B is the efficient route, and we refer to this as \emph{\underline{compressive sampling DMD}}.

\begin{figure*}
\begin{center}
\vspace{-.05in}
\begin{overpic}[width=.9\textwidth]{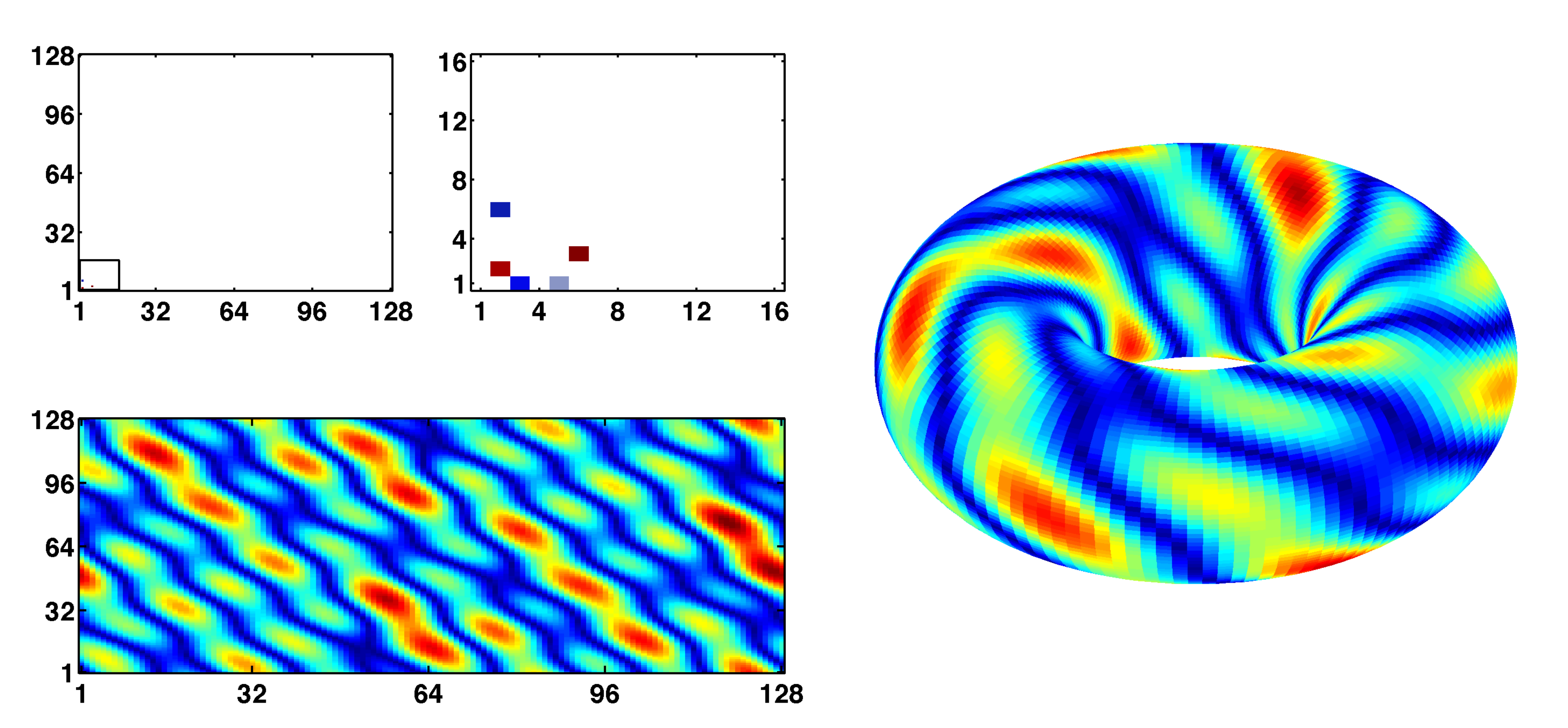}
{\scriptsize 
\put(31.5,33.75){\textbf{1}}
\put(31.5,30){\textbf{2}}
\put(32.75,29){\textbf{3}}
\put(35.25,29){\textbf{4}}
\put(36.5,31){\textbf{5}}
}
\put(21,21){\bf Spatial modes}
\put(19,44){\bf Fourier coefficients}
\put(70,39){\bf Spatial modes}
\end{overpic}
\vskip -.0in
\caption{Illustration of the dynamical system in Example 1.  $K=5$ Fourier coefficients are driven with a random linear dynamical system, resulting in the spatial dynamics shown.  Fifteen point sensors are placed randomly in space.  For video, see  {http://faculty.washington.edu/sbrunton/simulations/csdmd/Torus.mp4}.}\label{fig:fft1}
\end{center}
\end{figure*}

\begin{figure*}
\begin{center}
\begin{overpic}[width=.95\textwidth]{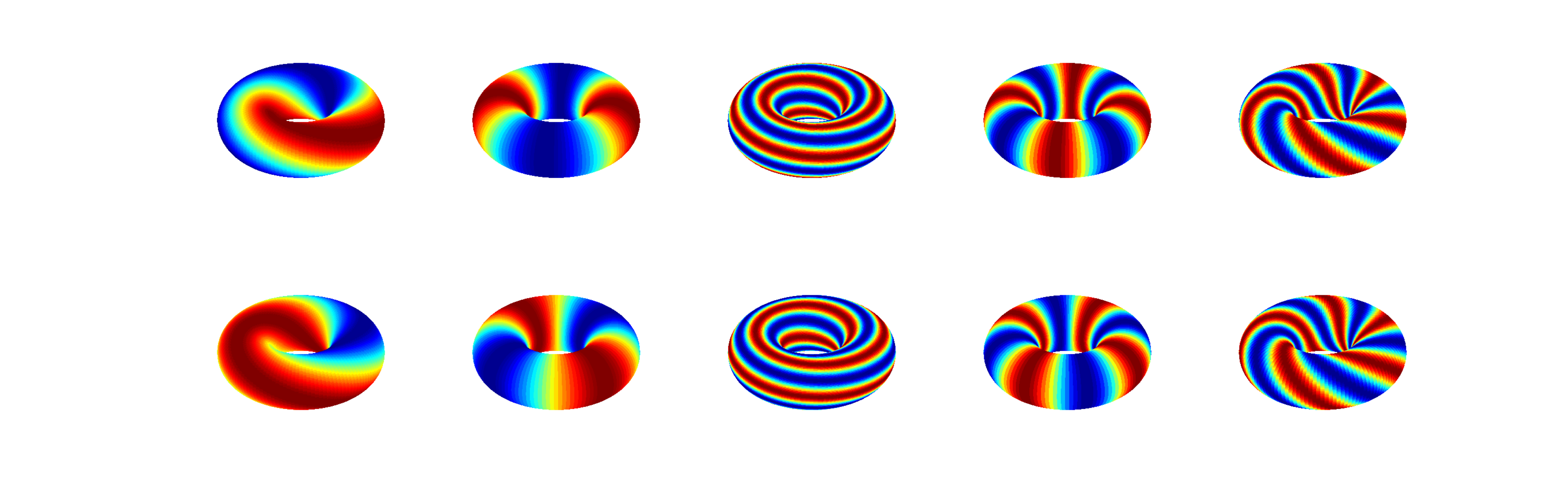}
\put(16,29){Mode 1}
\put(32.5,29){Mode 2}
\put(48.5,29){Mode 3}
\put(65,29){Mode 4}
\put(81,29){Mode 5}
\put(2,24){Real}
\put(2,9){Imaginary}
\end{overpic}
\vskip -.2in
\caption{Spatial-temporal coherent modes corresponding to non-zero Fourier modes.}\label{fig:fft2}
\end{center}
\end{figure*}
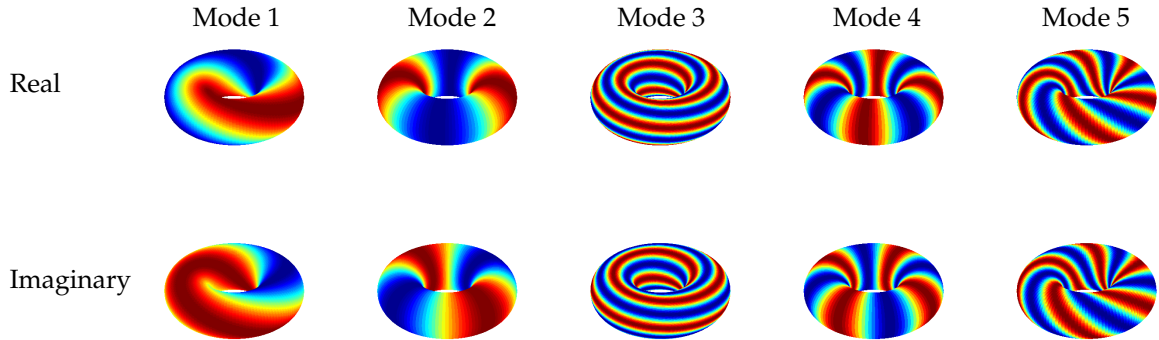

\section{Results}\label{sec:results}
We illustrate the above methods on two example problems that are relevant for fluid dynamics.  In the first example, we construct a spatially evolving system that has sparse dynamics in the Fourier domain.  In the second example, we consider the time-varying double-gyre, which has been used as a model for ocean mixing.

\subsection{Example 1: Sparse linear system in Fourier domain}
This system is designed to test to compressive DMD algorithms in a well-controlled numerical experiment.  We impose sparsity by creating a system with $K=5$ non-zero 2D spatial Fourier modes; all other modes are exactly zero.  It is also possible to allow the other Fourier modes to be contaminated with Gaussian noise and impose a very fast stable dynamic in each of these directions.  
We then define a stable linear, time-invariant dynamical system on the $K$ modes.  This is done by randomly choosing a temporal oscillation frequency and small but stable damping rate for each of the modes independently.  In this way, we construct a system in the spatial domain that is a linear combination of coherent spatial Fourier modes that each oscillate at a different fixed frequency.  Figure~\ref{fig:fft1} shows a snapshot of this system at $t=2$.  We see the five large Fourier mode coefficients generate distinct spatial coherent patterns.  Figure~\ref{fig:fft2} shows the five Fourier modes (real and imaginary parts) that contribute to the spatial structures in Figure~\ref{fig:fft1}.

This example is constructed to be an ideal test case for compressed sampling dynamic mode decomposition.  The linear time-invariant system underlying these dynamics is chosen at random, so the data matrix $\bX$ will contain significant energy from many of the modes.  Therefore, POD does not separate the spatial modes, as shown in Figure~\ref{fig:pod}.  Since each of our Fourier modes is oscillating at a fixed and distinct frequency, this is ideal for dynamic mode decomposition, which isolates the spatially coherent Fourier modes exactly, as shown in Figure~\ref{fig:dmd}. 

\begin{figure*}
\begin{center}
\begin{overpic}[width=.95\textwidth]{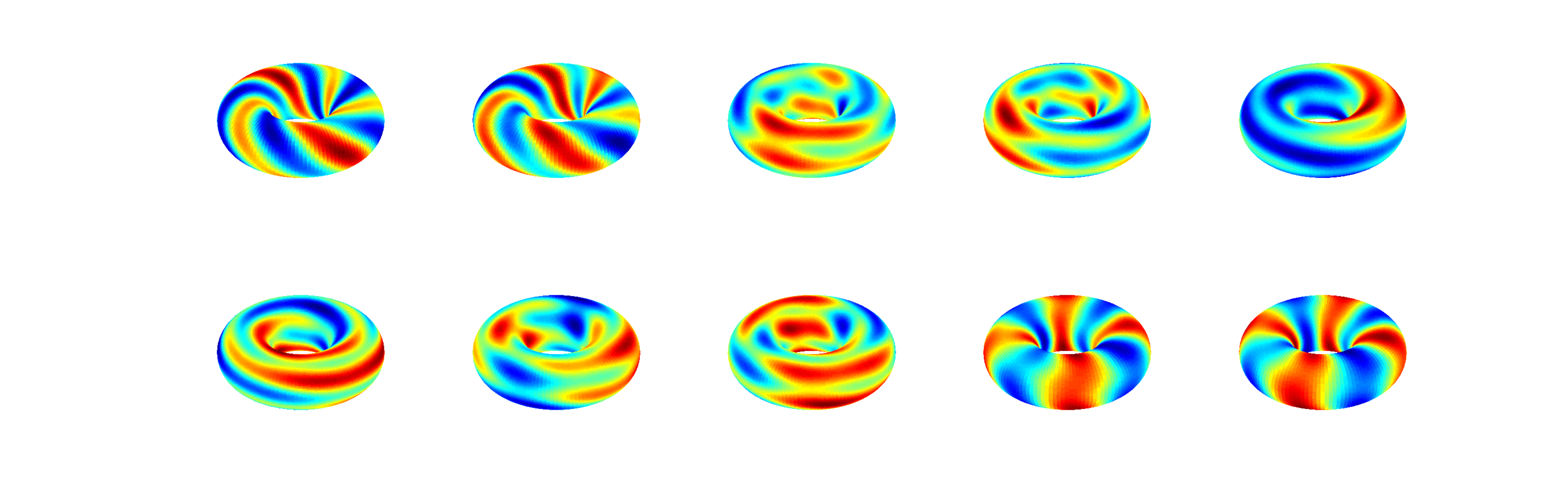}
\put(16,29){Mode 1}
\put(32.5,29){Mode 2}
\put(48.5,29){Mode 3}
\put(65,29){Mode 4}
\put(81,29){Mode 5}
\put(16,14){Mode 6}
\put(32.5,14){Mode 7}
\put(48.5,14){Mode 8}
\put(65,14){Mode 9}
\put(81,14){Mode 10}
\end{overpic}
\vskip -.2in
\caption{POD modes obtained from data.  Energetic modes mix the underlying Fourier modes.}\label{fig:pod}
\end{center}
\end{figure*}

\begin{figure*}
\begin{center}
\begin{overpic}[width=.95\textwidth]{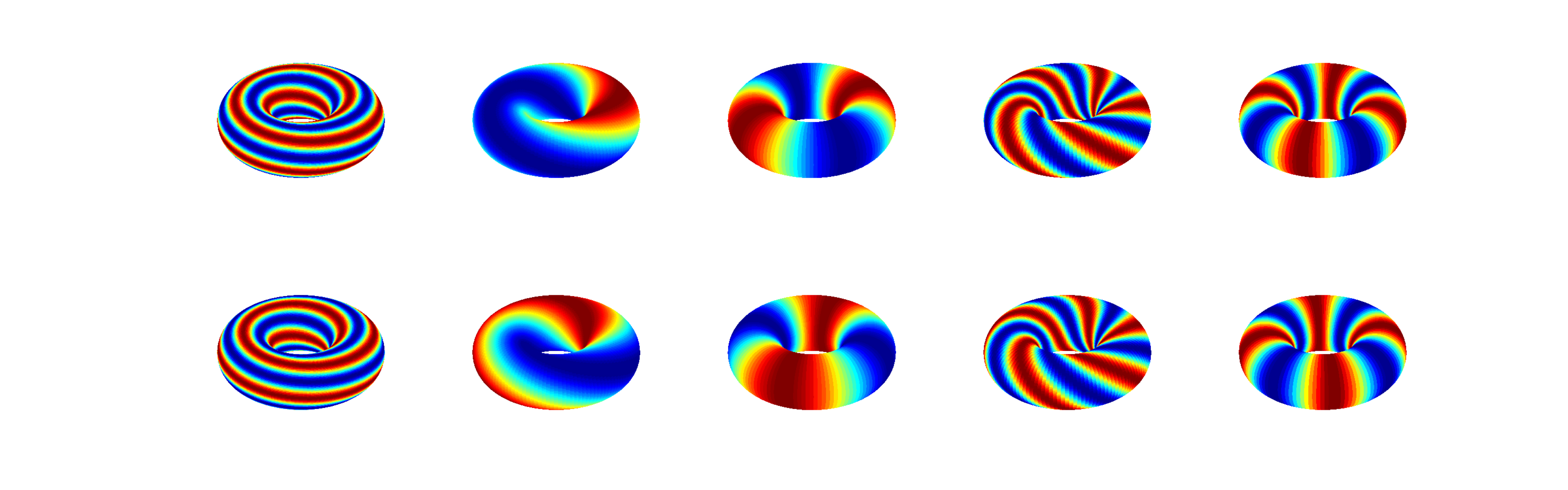}
\put(16,29){Mode 1}
\put(32.5,29){Mode 2}
\put(48.5,29){Mode 3}
\put(65,29){Mode 4}
\put(81,29){Mode 5}
\put(2,24){Real}
\put(2,9){Imaginary}
\end{overpic}
\vskip -.2in
\caption{DMD modes correctly isolate spatially coherent modes.  }\label{fig:dmd}
\end{center}
\end{figure*}

\begin{figure}
\begin{center}
\hspace*{-.25in}
\begin{overpic}[width=.56\textwidth]{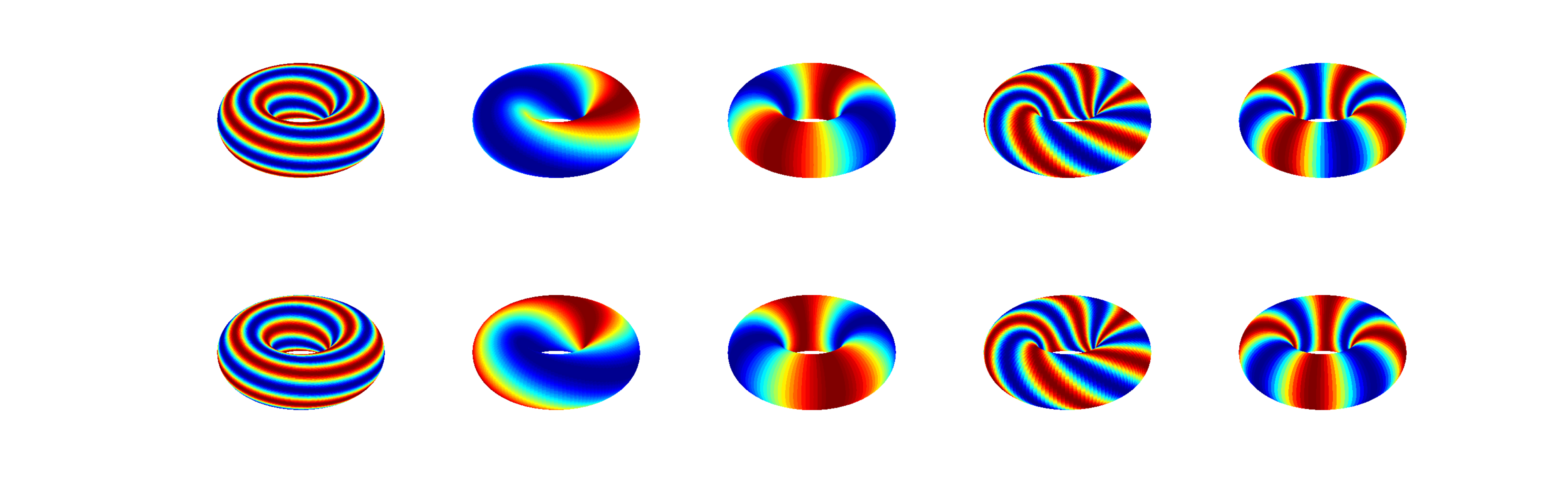}
\small
\put(14,29){Mode 1}
\put(30.5,29){Mode 2}
\put(46.5,29){Mode 3}
\put(63,29){Mode 4}
\put(79,29){Mode 5}
\end{overpic}
\vspace{-.25in}
\caption{DMD modes from compressed data, using ${\bPhi_{\bX}=\bX'\bV\bSigma^{-1}\bW_{\bY}}$.  (Path 1B, Sec.~\ref{sec:pathways})}\label{fig:compresseddmdFFT}
\end{center}
\end{figure}

\begin{figure}
\begin{center}
\hspace*{-.4in}
\begin{overpic}[width=.6\textwidth]{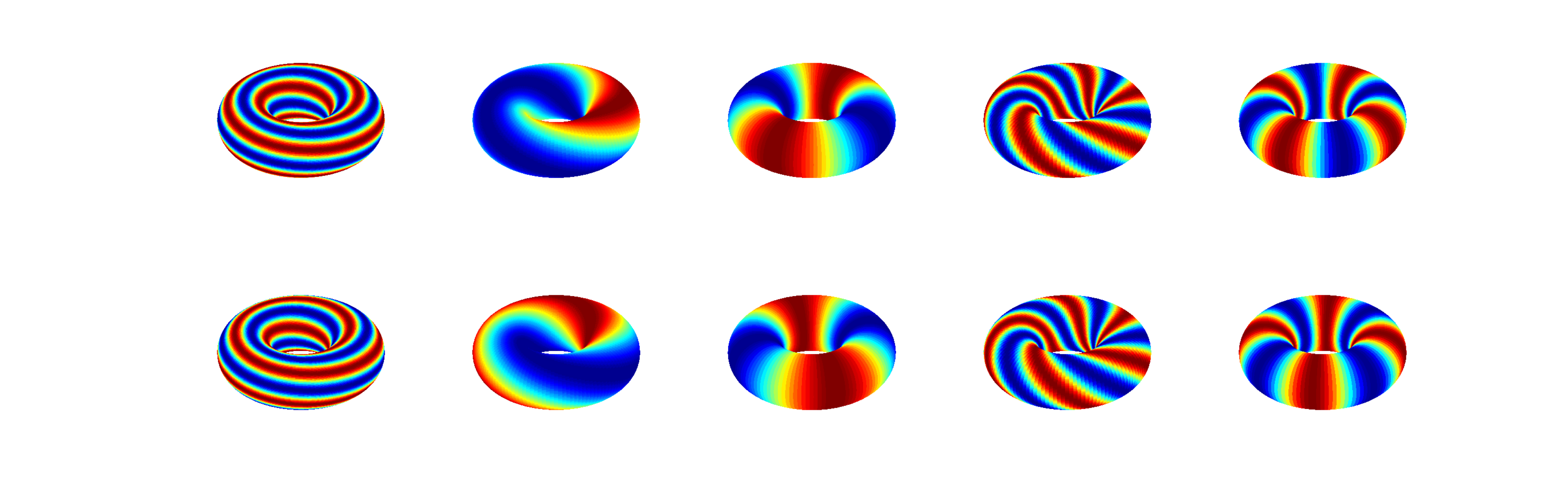}
\small
\put(14,29){Mode 1}
\put(30.5,29){Mode 2}
\put(46.5,29){Mode 3}
\put(63,29){Mode 4}
\put(79,29){Mode 5}
\end{overpic}
\vspace{-.25in}
\caption{Compressive-sampling DMD modes, using matching pursuit.  (Path 2B, Sec.~\ref{sec:pathways})}\label{fig:csdmdFFT}
\vspace{-.2in}
\end{center}
\end{figure}

\begin{figure}
\vspace{-.1in}
\begin{center}
\begin{overpic}[width=.5\textwidth]{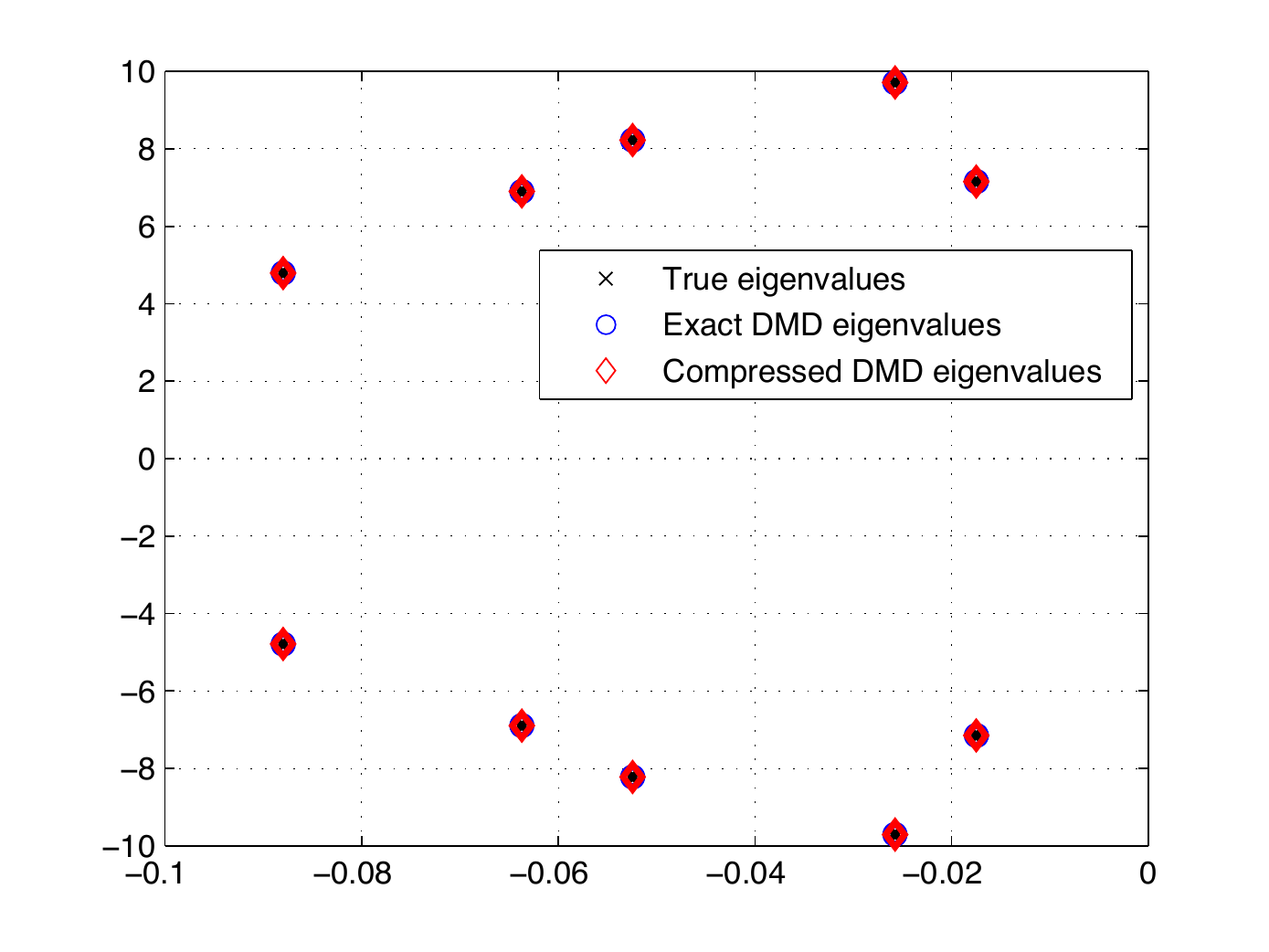}
\put(45,0){Real}
\put(5,30){\begin{sideways}Imaginary\end{sideways}}
\put(15,64){$\mathbb{C}$}
\end{overpic}
\caption{DMD modes captures dynamics faithfully.  }\label{fig:dmdeigs}
\vspace{-.25in}
\end{center}
\end{figure}

In the case of no background noise, the compressive sampling DMD algorithm (Path 2B in Sec.~\ref{sec:pathways}) works extremely well, as seen in the mode reconstruction in Figures~\ref{fig:csdmdFFT}.  Additionally, the method of compressed DMD (Path 1B in Sec.~\ref{sec:pathways}) starting with full-state snapshots, compressing, performing DMD, and then reconstructing using formula in Eq.~(\ref{eq:compressedDMD}) results in accurate reconstruction, shown in Figure~\ref{fig:compresseddmdFFT}.  Both compressive-sampling DMD and compressed DMD match the true eigenvalues nearly exactly, as shown in Figure~\ref{fig:dmdeigs}.

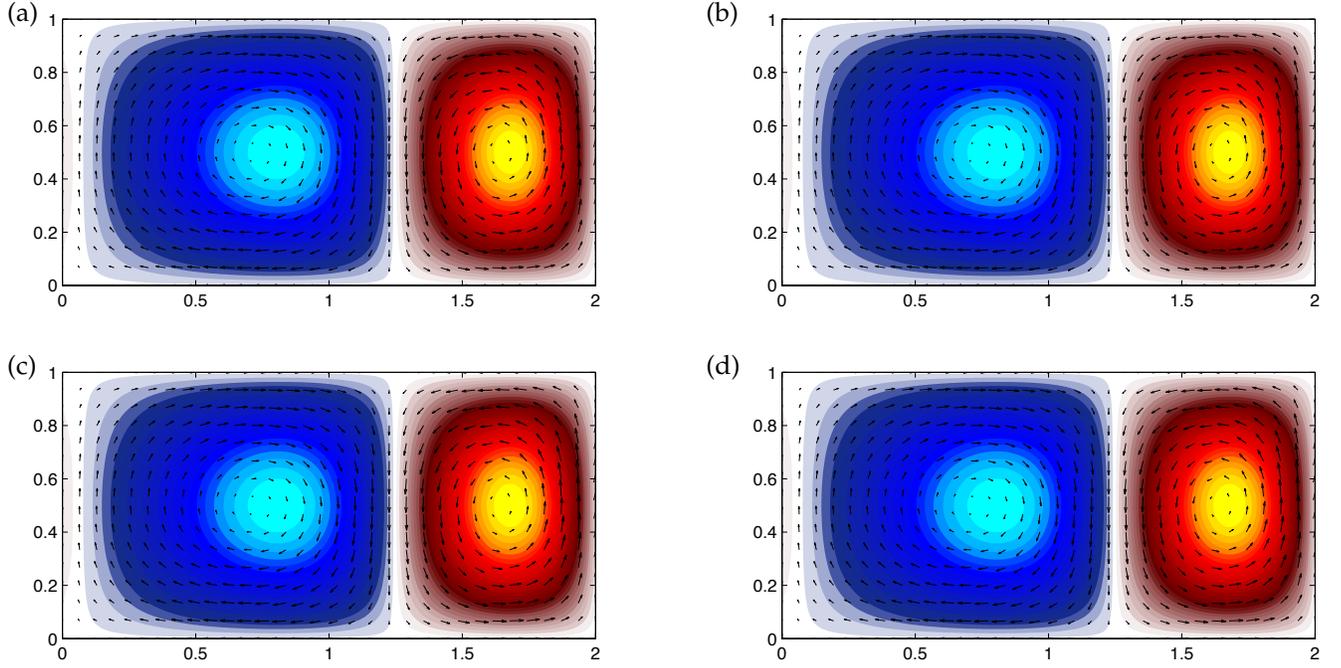
\begin{figure*}
\begin{center}
\begin{tabular}{cc}
\begin{overpic}[width=0.5\textwidth]{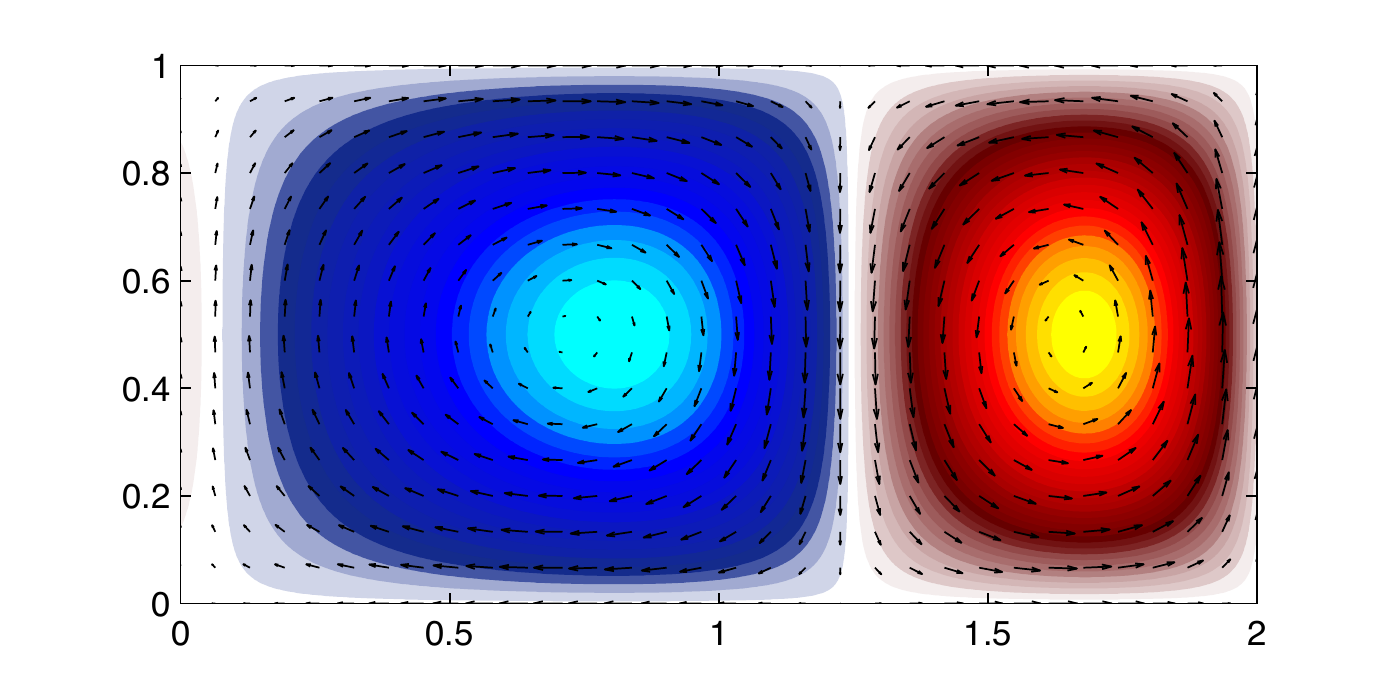}
\put(5,45){(a)}
\end{overpic}
&
\begin{overpic}[width=0.5\textwidth]{f_figure10a}
\put(2,45){(b)}
\end{overpic}\\
\begin{overpic}[width=0.5\textwidth]{f_figure10a}
\put(5,45){(c)}
\end{overpic}
&
\begin{overpic}[width=0.5\textwidth]{f_figure10a}
\put(2,45){(d)}
\end{overpic}
\end{tabular}
\caption{(a)  Double gyre vector field overlaid onto of vorticity for $t=0.2$, $\epsilon=0.25$, $A=0.1$, $\omega = 2\pi/10$.  (b)  Logarithmically scaled Fourier coefficient magnitudes of the vorticity.  (c)  Compressed double gyre vorticity field using $1.0\%$ of the Fourier coefficients.  (d)  Top $1.0\%$ largest magnitude Fourier coefficients.}\label{fig:dg}
\vspace{-.2in}
\end{center}
\end{figure*}

The $\bX,\bX'$ data matrices are obtained on a $128\times 128$ spatial grid from times $0$ to $2$ in increments of $\Delta t = 0.01$, so that $\bX\in\mathbb{R}^{16384\times 200}$.  We use $p=15$ measurements; single-pixel, Gaussian, and Bernouli random measurements all yield correct results.

When we add small to moderate amounts of background noise (2-5\% RMS) in the Fourier domain, a number of things change.  The modes and frequencies are still very well characterized by both methods of compressed DMD and compressive-sampling DMD.  However, the resulting DMD damping is often too large; the disagreement is more significant when the full-state DMD eigenvalues are over-damped compared with exact eigenvalues.  Recent results predict the change in DMD spectrum with small additive noise~\cite{Bagheri:2013b}, and our results appear to be consistent.  More work is needed to characterize and address the issue of noise and DMD.

\subsection{Example 2: Double gyre flow}
The double gyre is a simple two-dimensional model~\cite{Solomon:1988} that is often used to study mixing between ocean basins. The double-gyre flow is given by the following stream-function
\begin{align}
\begin{split}
\psi(x,y,t)&=A\sin\left(\pi f(x,t)\right)\sin(\pi y)\\
f(x,t)&=\epsilon\sin(\omega t)x^2+x-2\epsilon\sin(\omega t) x,
\end{split}
\end{align}
which results in the following time-periodic vector field
\begin{align}
\begin{split}
u&=-\frac{\partial \psi}{\partial y}=-\pi A\sin\left(\pi f(x)\right)\cos(\pi y)\\
v&=\frac{\partial \psi}{\partial x}=\pi A\cos\left(\pi f(x)\right)\sin(\pi y)\frac{df}{dx}
\end{split}
\label{eqn:gyre}
\end{align}
on the closed and bounded domain $[0,2]\times[0,1]$.  Typical parameter values are $A=0.1, \omega=2\pi/10, \epsilon=0.25$.

Figure~\ref{fig:dg} illustrates the double gyre vector field, with color representing vorticity and arrows showing the vector field.  The vorticity field is highly compressible, in that 99\% of the Fourier coefficients may be zeroed with little effect on the reconstructed vorticity field (panel c). 

\begin{figure*}
\vspace{-.1in}
\begin{tabular}{cc}
\begin{overpic}[width=0.5\textwidth]{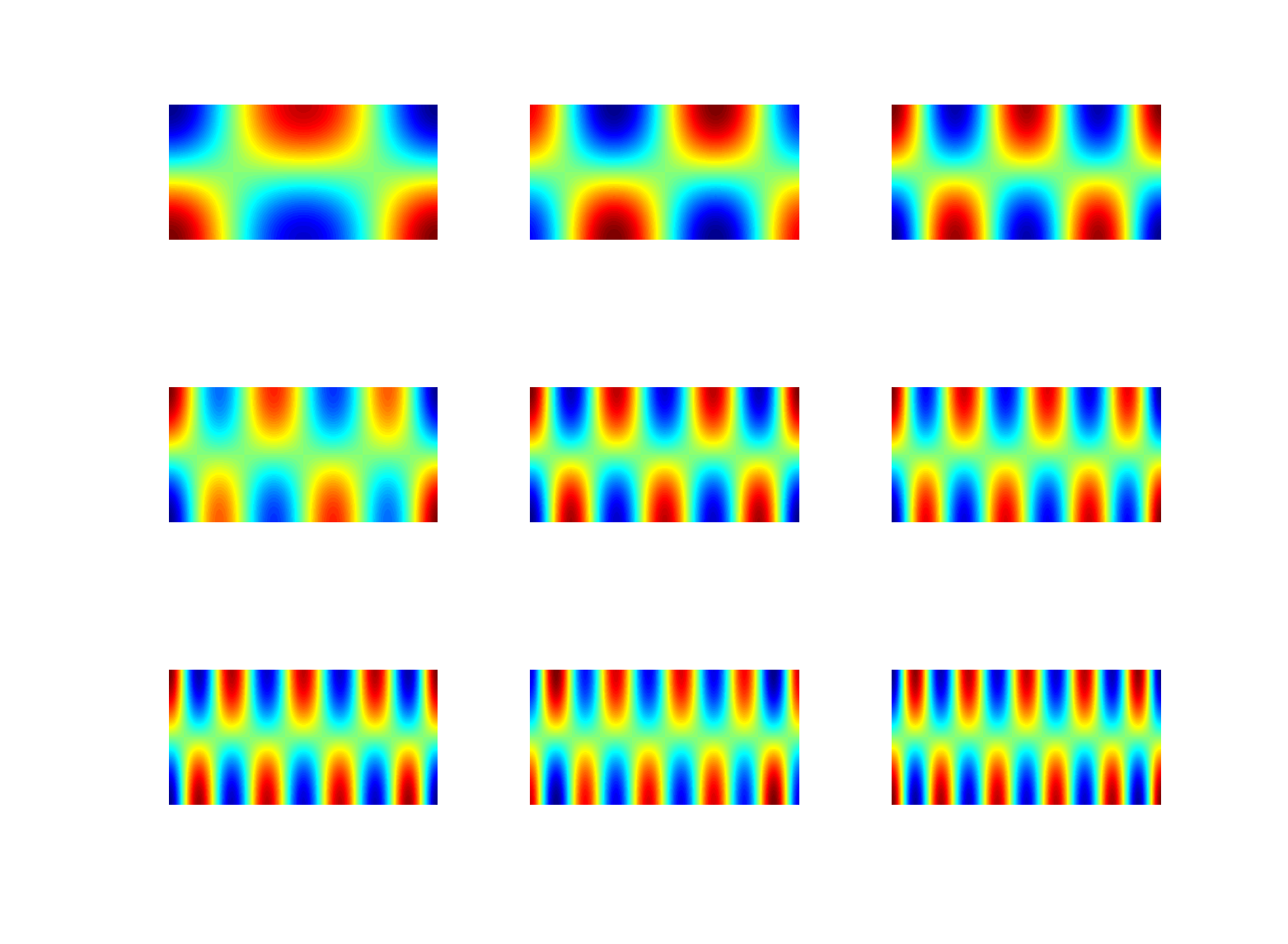}
\put(5,63){(a)}
\small
\put(18,66.5){Mode 1}
\put(46,66.5){Mode 2}
\put(74,66.5){Mode 3}
\put(18,44.5){Mode 4}
\put(46,44.5){Mode 5}
\put(74,44.5){Mode 6}
\put(18,22.5){Mode 7}
\put(46,22.5){Mode 8}
\put(74,22.5){Mode 9}
\end{overpic}
& \hspace{-.4in}
\begin{overpic}[width=0.5\textwidth]{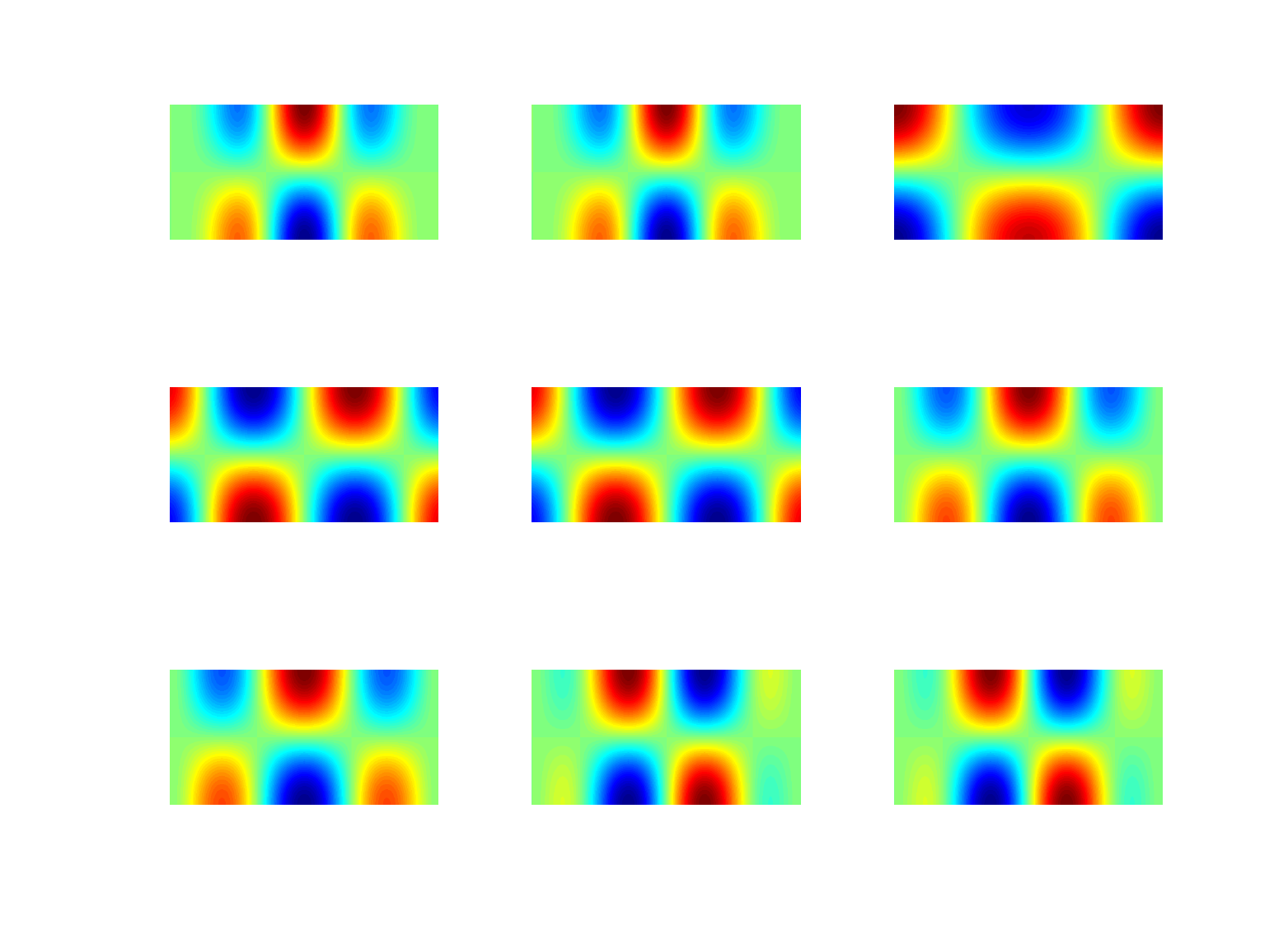}
\put(5,63){(b)}
\small
\put(18,66.5){Mode 1}
\put(46,66.5){Mode 2}
\put(74,66.5){Mode 3}
\put(18,44.5){Mode 4}
\put(46,44.5){Mode 5}
\put(74,44.5){Mode 6}
\put(18,22.5){Mode 7}
\put(46,22.5){Mode 8}
\put(74,22.5){Mode 9}
\end{overpic}
\vspace{-.3in}
\\
\begin{overpic}[width=0.5\textwidth]{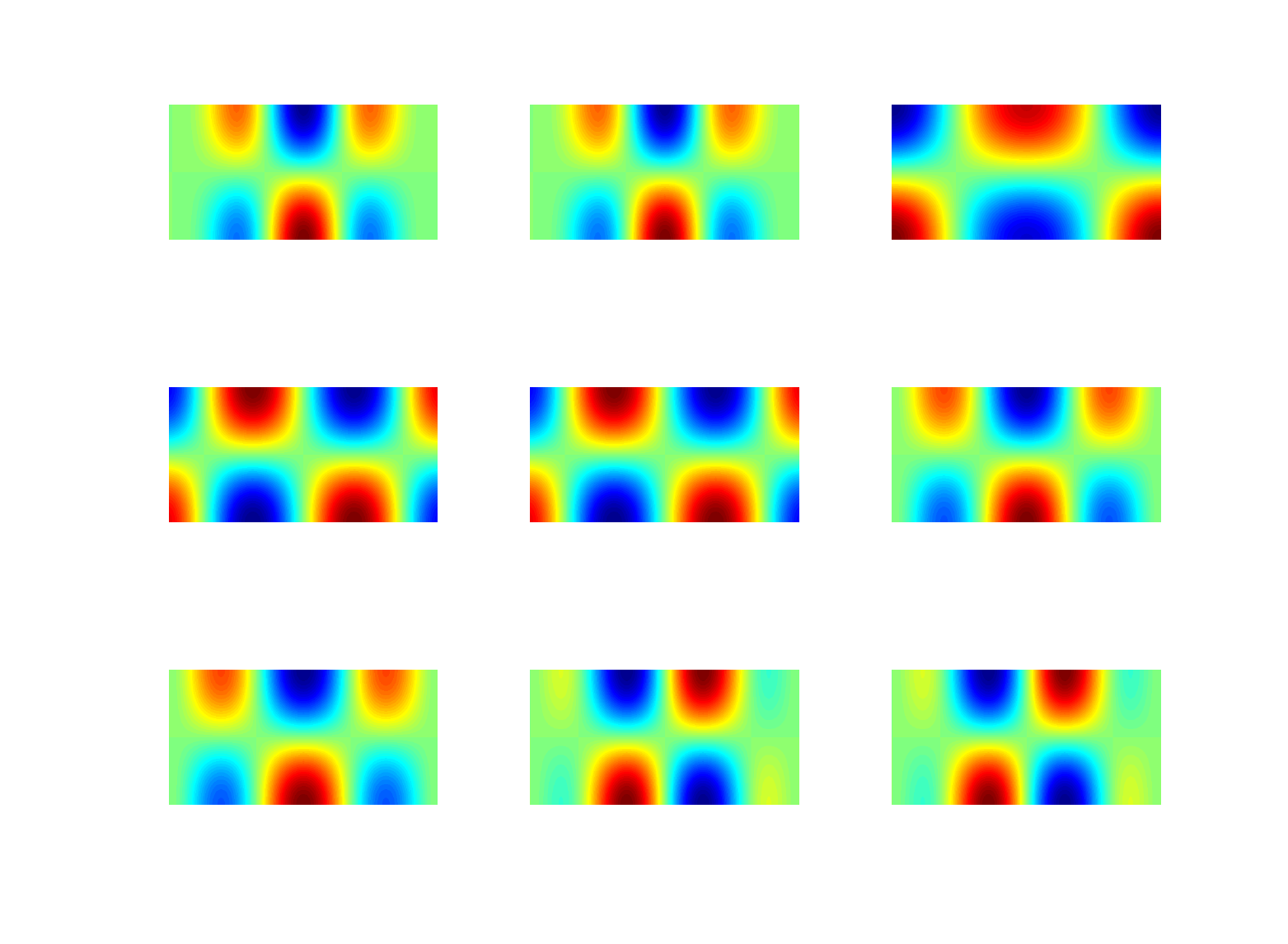}
\put(5,63){(c)}
\small
\put(18,66.5){Mode 1}
\put(46,66.5){Mode 2}
\put(74,66.5){Mode 3}
\put(18,44.5){Mode 4}
\put(46,44.5){Mode 5}
\put(74,44.5){Mode 6}
\put(18,22.5){Mode 7}
\put(46,22.5){Mode 8}
\put(74,22.5){Mode 9}
\end{overpic}
& \hspace{-.4in}
\begin{overpic}[width=0.5\textwidth]{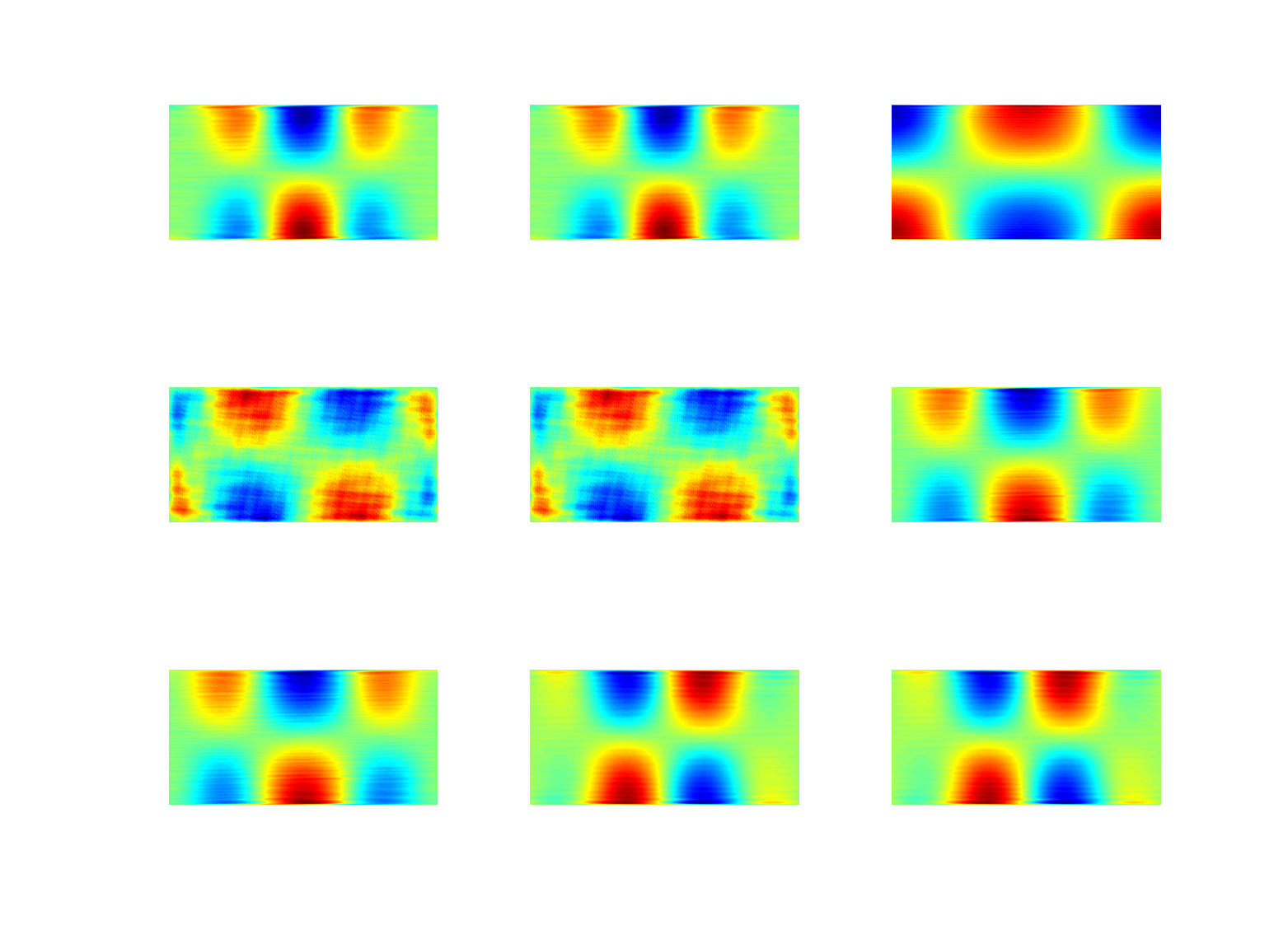}
\put(5,63){(d)}
\small
\put(18,66.5){Mode 1}
\put(46,66.5){Mode 2}
\put(74,66.5){Mode 3}
\put(18,44.5){Mode 4}
\put(46,44.5){Mode 5}
\put(74,44.5){Mode 6}
\put(18,22.5){Mode 7}
\put(46,22.5){Mode 8}
\put(74,22.5){Mode 9}
\end{overpic}
\end{tabular}
\vspace{-.2in}
\caption{Modal decompositions for the double gyre: (a) POD modes, (b) DMD modes, (c) compressed DMD modes (Path 1B, Sec.~\ref{sec:pathways}), and (d) compressive-sampling DMD modes using matching pursuit (Path 2B, Sec.~\ref{sec:pathways}).}\label{fig:dgDMD}
\end{figure*}
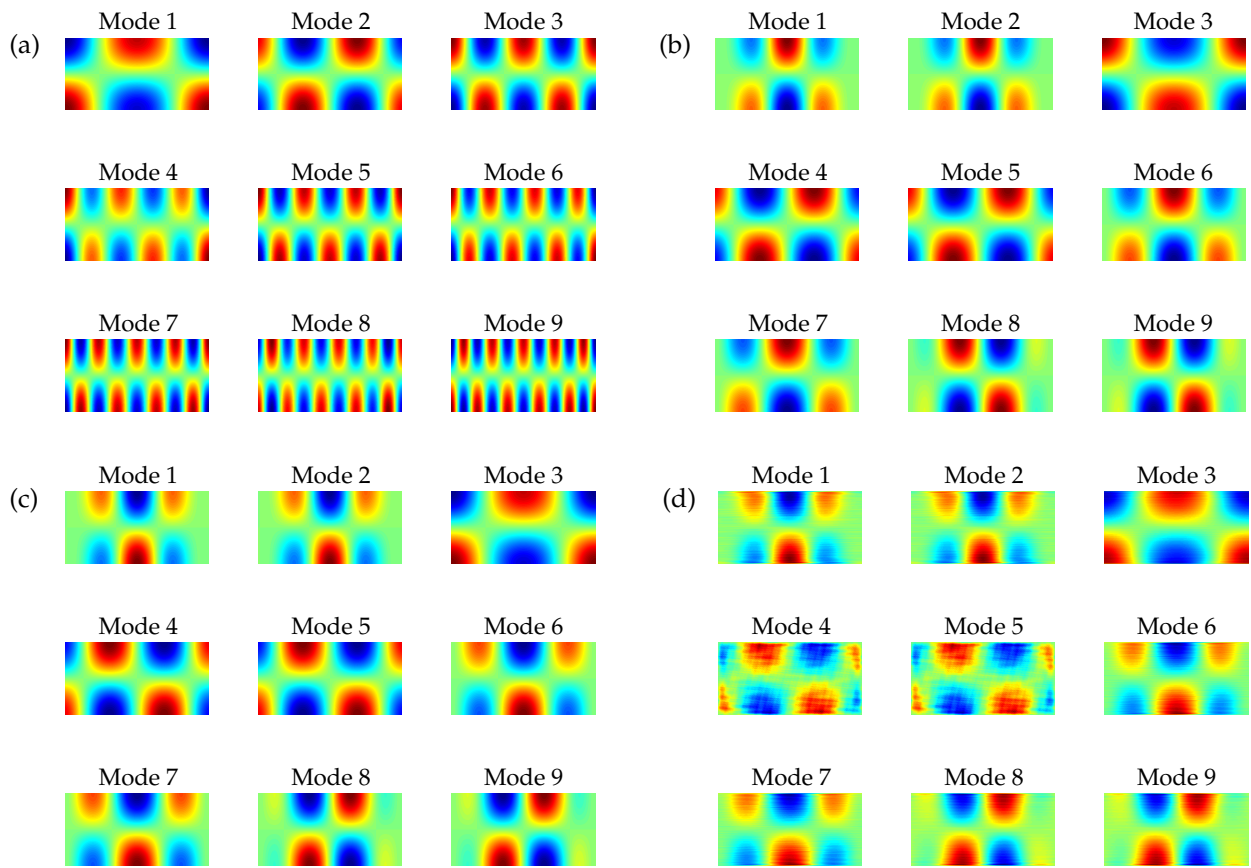

The datasets $\bX,\bX'$ are constructed for a $512\times 256$ spatial grid at times $0$ to $15$ with $\Delta t=0.1$.  The POD modes for the double gyre data set are shown in Figure~\ref{fig:dgDMD} (a).  The DMD modes are shown in Figure~\ref{fig:dgDMD} (b).  Taking 2500 single pixel measurements, which accounts for under 2\% of the total pixels, the reconstructed modes are shown in Figure~\ref{fig:dgDMD} (c) and (d).  The modes exhibit good agreement.  Moreover, the DMD eigenvalues are exactly recovered by compressive DMD, shown in Figure~\ref{fig:DGdmdeigs}.  

\begin{figure}
\begin{center}
\begin{overpic}[width=.45\textwidth]{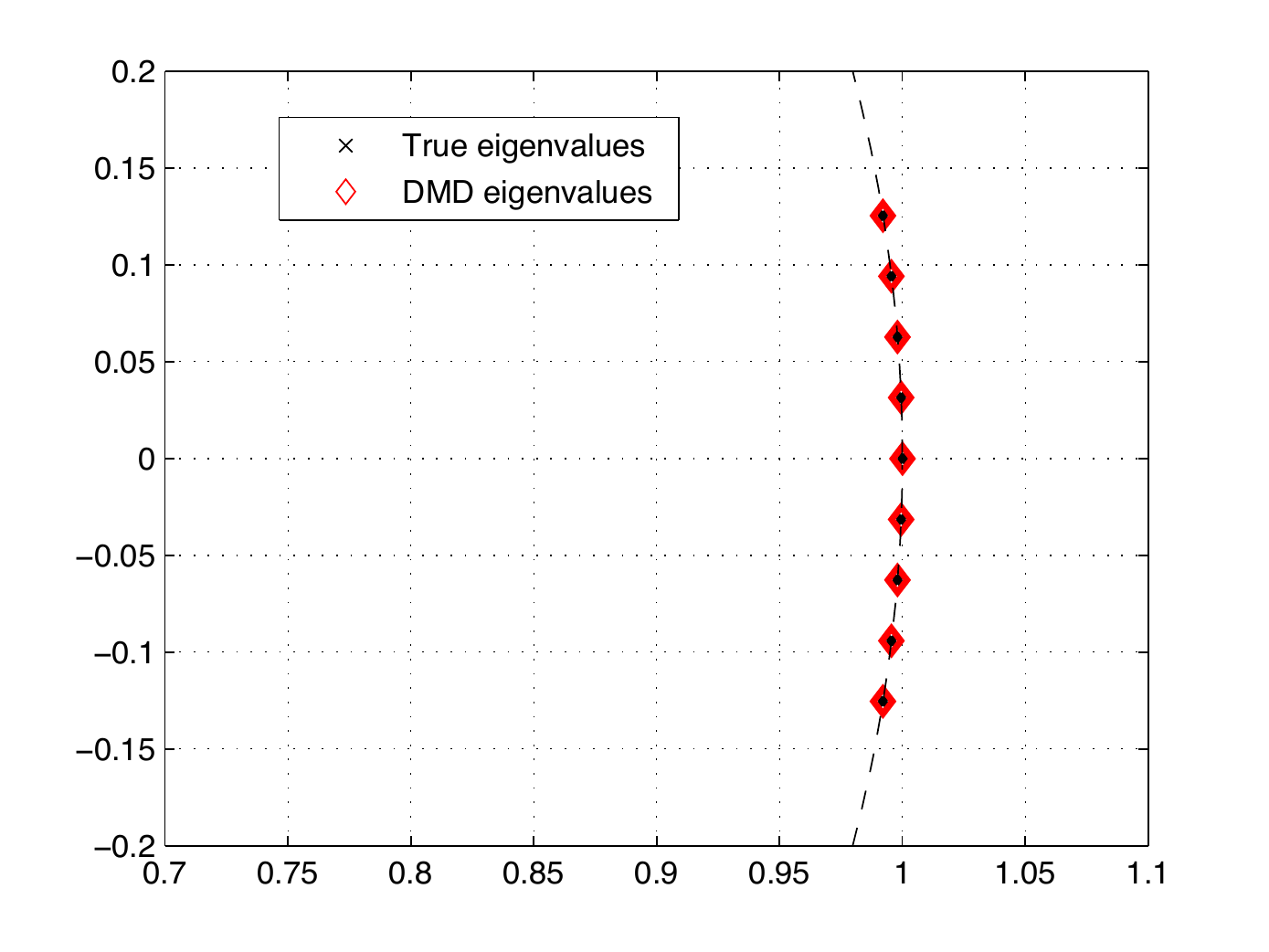}
\put(45,0){Real}
\put(1,30){\begin{sideways}Imaginary\end{sideways}}
\put(15,64){$\mathbb{C}$}
\end{overpic}
\vspace{0.15in}
\caption{DMD modes captures dynamics faithfully.  }\label{fig:DGdmdeigs}
\vspace{-.35in}
\end{center}
\end{figure}

\section{Discussion}\label{sec:discussion}
In this work we have developed a method that leverages compressive sampling to compute the dynamic mode decomposition (DMD) for spatially subsampled or projected data.  There are two key advances resulting from this work.  First, it is possible to reconstruct full-state DMD modes from heavily subsampled output-projected data using compressive sampling.  This is called \emph{compressive-sampling DMD}, and is illustrated in Sec.~\ref{sec:pathways} and Figure~\ref{fig:schematic} as Path 2B.  Second, if full-state snapshots are available, it is possible to first compress the data, perform DMD, and then reconstruct by taking a linear combination of the snapshot data, determined by the projected DMD.  This is denoted by \emph{compressed DMD}, and is explained in Path 1B in Sec.~\ref{sec:pathways} and Figure~\ref{fig:schematic}.  The theory relies on relationships between full-state and projected DMD established in Sec.~\ref{sec:csdmd:1}.  We also show that DMD is invariant to left and right unitary transformations.  We then relax this condition and use the restricted isometry principle that is satisfied when incoherent measurements are applied to a signal that is sparse in some basis.   

Both of these methods are demonstrated to be effective on two example problems with relevance to fluid dynamics and oceanographic/atmospheric sciences.  In the first example, a low-order linear dynamical system is evolved on a few Fourier coefficients, establishing a high-dimensional, time-varying spatial flow.  In the second example, we consider the double gyre flow, which is a model for ocean mixing.  Both examples demonstrate efficient, accurate reconstruction of the DMD eigenvalues and modes from few spatial measurements.  

There are a number of interesting directions that arise from this work.  First, it will be a natural extension to apply these methods to high-dimensional systems in fluid dynamics and to oceanographic/atmospheric measurements.  The reduced burden of spatial sampling may also allow for increased temporal sampling rates in particle image velocimetry (PIV), similar to recent advances in MRI~\cite{Sankaranarayanan:2012,Patel:2013,Shi:2014}.  In PIV, the data transfer from camera to RAM is the limiting factor, although one may envision compressing the full-state data before transferring to memory.  It is also important to characterize and address the effect of noise on DMD, both the full-state and compressed versions~\cite{Bagheri:2013b}.  Finally, it may be possible to combine the spatial compressive sampling strategy advocated here with the temporal sampling strategy in~\cite{Tu:2014b} for greater efficiency gains.

\section*{Appendix: SVD and Unitary transformations}\label{app:svd}
In this appendix we demonstrate that left or right-multiplication of a data matrix $\bX$ by a unitary transformation preserves all of the terms in the singular value decomposition except for the corresponding left or right unitary matrix $\bU$ or $\bV$, respectively.  These matrices are simply multiplied by the new unitary transformation.  

\subsection{Method of snapshots}
Typically when computing the SVD of a large data matrix $\bX\in\mathbb{R}^{n\times m}$, with $n\gg m$, we use the method of snapshots~\cite{Sirovich:1987}.  In this method, we compute the $m\times m$ matrix $\bX^*\bX$ and solve the following eigendecomposition:
\begin{align*}
&\bX^*\bX=\bV\bSigma^2\bV^*\\
\Longrightarrow\quad &\bX^*\bX\bV = \bV\bSigma^2
\end{align*}
With $\bV$ and $\bSigma$ computed, it is possible to construct $\bU$:
\begin{align*}
\bU = \bX\bV\bSigma^{-1}
\end{align*}
Thus, we have the singular value decomposition: ${\bX=\bU_{\bX}\bSigma_{\bX}\bV_{\bX}^*}$.  We have added the subscript $\bX$ to denote that the $\bU,\bSigma$, and $\bV$ are computed from $\bX$.

\subsection{Left unitary transformation}
Now consider the data matrix $\bY=\bC\bX$, where $\bC$ is unitary.  We find that
\begin{align*}
\bY^*\bY&=\bX^*\bC^*\bC\bX=\bX^*\bX,
\end{align*}
so the projected data has the same eigendecompositon and results in the same $\bV_{\bX}$ and $\bSigma_{\bX}$.  Now, constructing $\bU_{\bY}$, we have:
\begin{align*}
\bU_{\bY}=\bY\bV_{\bX}\bSigma_{\bX}^{-1} = \bC\bX \bV_{\bX}\bSigma_{\bX}^{-1}=\bC\bU_{\bX}.
\end{align*}
Therefore, we have the following singular value decomposition: $\bY=\bC\bU_{\bX}\bSigma_{\bX}\bV_{\bX}^*$.  
In other words, $\bC\bU_{\bX}$ are the new left-singular vectors.

\subsection{Right unitary transformation}
Similarly, consider $\bY=\bX\bP^*$, where $\bP$ is a unitary matrix.  We fine
\begin{align*}
\bY^*\bY = \bP\bX^*\bX\bP^* = \bP\bV_{\bX}\bSigma_{\bX}^2\bV_{\bx}^*\bP^*.
\end{align*}
This results in the following eigendecomposition:
\begin{align*}
\bY^*\bY \bP\bV_{\bX} = \bP_{\bX}\bV_{\bX} \bSigma_{\bX}^2.
\end{align*}
Therefore, $\bV_{\bY}=\bP\bV_{\bX}$ and $\bSigma_{\bY}=\bSigma_{\bX}$.  It is then possible to construct $\bU_{\bY}$:
\begin{align*}
\bU_{\bY} = \bY\bP\bV_{\bX}\bSigma_{\bX}^{-1} = \bX\bV_{\bX}\bSigma_{\bX}^{-1} = \bU_{\bX}.
\end{align*}
The singular value decomposition ${\bY=\bU_{\bX}\bSigma_{\bX}\bV_{X}^*\bP^*}$ has $\bP\bV_{\bX}$ as the new right-singular vectors.

\section*{Acknowledgements}
J. N. Kutz acknowledges support from the  U.S. Air Force Office of Scientific Research (FA9550-09-0174).  
J. L. Proctor acknowledges support by Intellectual Ventures Laboratory.  
The authors thank Jonathan Tu for discussion on both sparsity and dynamic mode decomposition.  We also thank Bingni Brunton for discussions on compressive sampling and for suggesting greedy algorithms. 


\begin{spacing}{.89}
\footnotesize{
\bibliographystyle{plain}
\bibliography{csdmd}

\begin{thebibliography}{10}

\bibitem{MarsdenMTAA}
R.~Abraham, J.~E. Marsden, and T.~Ratiu.
\newblock {\em Manifolds, Tensor Analysis, and Applications}, volume~75 of {\em
  Applied Mathematical Sciences}.
\newblock Springer-Verlag, 1988.

\bibitem{Bagheri:2013b}
S.~Bagheri.
\newblock Effects of small noise on the {DMD}/{K}oopman spectrum.
\newblock {\em Bulletin Am.~Phys.~Soc.}, 58(18):H35.0002, p.~230, 2013.

\bibitem{Bagheri:2013}
S.~Bagheri.
\newblock Koopman-mode decomposition of the cylinder wake.
\newblock {\em Journal of Fluid Mechanics}, 726:596--623, 2013.

\bibitem{Glauser:2013}
Z.~Bai, T.~Wimalajeewa, Z.~Berger, G.~Wang, M.~Glauser, and P.~K. Varshney.
\newblock Physics based compressive sensing approach applied to airfoil data
  collection and analysis.
\newblock AIAA Paper 2013-0772, 51st Aerospace Sciences Meeting, January 2013.

\bibitem{Baraniuk:2007}
R.~G. Baraniuk.
\newblock Compressive sensing.
\newblock {\em IEEE Signal Processing Magazine}, 24(4):118--120, 2007.

\bibitem{Baraniuk:2009}
R.~G. Baraniuk, V.~Cevher, M.~F. Duarte, and C.~Hegde.
\newblock Model-based compressive sensing.
\newblock {\em IEEE Transactions on Information Theory}, 56(4):1982--2001,
  2010.

\bibitem{Berkooz:1993}
G.~Berkooz, P.~Holmes, and J.~L. Lumley.
\newblock The proper orthogonal decomposition in the analysis of turbulent
  flows.
\newblock {\em Annual Review of Fluid Mechanics}, 23:539--575, 1993.

\bibitem{Bright:2013}
I.~Bright, G.~Lin, and J.~N. Kutz.
\newblock Compressive sensing and machine learning strategies for
  characterizing the flow around a cylinder with limited pressure measurements.
\newblock {\em Physics of Fluids}, 25:127102--1--127102--15, 2013.

\bibitem{Brunton:2014a}
B.~W. Brunton, S.~L. Brunton, J.~L. Proctor, and J.~N. Kutz.
\newblock Optimal sensor placement and enhanced sparsity for classification.
\newblock {\em submitted for publication}, 2013.

\bibitem{Candes:2006}
E.~J. Cand\`es.
\newblock Compressive sensing.
\newblock {\em Proceedings of the International Congress of Mathematics}, 2006.

\bibitem{Candes:2006a}
E.~J. Cand\`es, J.~Romberg, and T.~Tao.
\newblock Robust uncertainty principles: exact signal reconstruction from
  highly incomplete frequency information.
\newblock {\em IEEE Transactions on Information Theory}, 52(2):489--509, 2006.

\bibitem{Candes:2006c}
E.~J. Cand\`es, J.~Romberg, and T.~Tao.
\newblock Stable signal recovery from incomplete and inaccurate measurements.
\newblock {\em Communications in Pure and Applied Mathematics}, 8(1207--1223),
  59.

\bibitem{Candes:2006b}
E.~J. Cand\`es and T.~Tao.
\newblock Near optimal signal recovery from random projections: Universal
  encoding strategies?
\newblock {\em IEEE Transactions on Information Theory}, 52(12):5406--5425,
  2006.

\bibitem{Chen:2012}
K.~K. Chen, J.~H. Tu, and C.~W. Rowley.
\newblock Variants of dynamic mode decomposition: Boundary condition,
  {K}oopman, and {F}ourier analyses.
\newblock {\em Journal of Nonlinear Science}, 22(6):887--915, 2012.

\bibitem{Donoho:2006}
D.~L. Donoho.
\newblock Compressed sensing.
\newblock {\em IEEE Transactions on Information Theory}, 52(4):1289--1306,
  2006.

\bibitem{Fowler:2009}
J.~E. Fowler.
\newblock Compressive-projection principal component analysis.
\newblock {\em IEEE Transactions on Image Processing}, 18(10):2230--2242, 2009.

\bibitem{Gilbert:2010}
A.~C. Gilbert and P.~Indyk.
\newblock Sparse recovery using sparse matrices.
\newblock {\em Proceedings of the IEEE}, 98(6):937--947, 2010.

\bibitem{Gilbert:2012}
A.~C. Gilbert, J.~Y. Park, and M.~B. Wakin.
\newblock Sketched {SVD}: Recovering spectral features from compressive
  measurements.
\newblock {\em ArXiv e-prints}, 2012.

\bibitem{Grosek:2013}
J.~Gosek and J.~N. Kutz.
\newblock Dynamic mode decomposition for real-time background/foreground
  separation in video.
\newblock {\em submitted for publication}, 2013.

\bibitem{Grilli:2012}
M.~Grilli, P.~J. Schmid, S.~Hickel, and N.~A. Adams.
\newblock Analysis of unsteady behaviour in shockwave turbulent boundary layer
  interaction.
\newblock {\em Journal of Fluid Mechanics}, 700:16--28, 2012.

\bibitem{kalman:1965}
B.~L. Ho and R.~E. Kalman.
\newblock Effective construction of linear state-variable models from
  input/output data.
\newblock In {\em Proceedings of the 3rd Annual Allerton Conference on Circuit
  and System Theory}, pages 449--459, 1965.

\bibitem{HLBR_turb}
P.~J. Holmes, J.~L. Lumley, G.~Berkooz, and C.~W. Rowley.
\newblock {\em Turbulence, coherent structures, dynamical systems and
  symmetry}.
\newblock Cambridge Monographs in Mechanics. Cambridge University Press,
  Cambridge, England, 2nd edition, 2012.

\bibitem{JL:1984}
W.~B Johnson and J.~Lindenstrauss.
\newblock Extensions of lipschitz mappings into a hilbert space.
\newblock {\em Contemporary mathematics}, 26(189-206):1, 1984.

\bibitem{Javanovic:2012}
M.~R. Jovanovi\'{c}, P.~J. Schmid, and J.~W. Nichols.
\newblock Low-rank and sparse dynamic mode decomposition.
\newblock {\em Center for Turbulence Research}, 2012.

\bibitem{ERA:1985}
J.~N. Juang and R.~S. Pappa.
\newblock An eigensystem realization algorithm for modal parameter
  identification and model reduction.
\newblock {\em Journal of Guidance, Control, and Dynamics}, 8(5):620--627,
  1985.

\bibitem{Kevrekidis:2003}
I.~G. Kevrekidis, C.~W. Gear, J.~M. Hyman, P.~G. Kevrekidis, O.~Runborg, and
  C.~Theodoropoulos.
\newblock Equation-free, coarse-grained multiscale computation: Enabling
  microscopic simulators to perform system-level analysis.
\newblock {\em Communications in Mathematical Science}, 1(4):715--762, 2003.

\bibitem{Koopman:1931}
B.~O. Koopman.
\newblock Hamiltonian systems and transformation in hilbert space.
\newblock {\em Proceedings of the National Academy of Sciences},
  17(5):315--318, 1931.

\bibitem{Kutz:2013}
J.~N. Kutz.
\newblock {\em Data-Driven Modeling \& Scientific Computation: Methods for
  Complex Systems \& Big Data}.
\newblock Oxford University Press, 2013.

\bibitem{Lumley:1970}
J.~L. Lumley.
\newblock {\em Stochastic Tools in Turbulence}.
\newblock Academic Press, 1970.

\bibitem{ERA:2009}
Z.~Ma, S.~Ahuja, and C.~W. Rowley.
\newblock Reduced order models for control of fluids using the eigensystem
  realization algorithm.
\newblock {\em Theoretical and Computational Fluid Dynamics}, 25(1):233--247,
  2011.

\bibitem{Mezic:2013}
I.~Mezi\'c.
\newblock Analysis of fluid flows via spectral properties of the {K}oopman
  operator.
\newblock {\em Annual Review of Fluid Mechanics}, 45:357--378, 2013.

\bibitem{Needell:2010}
D.~Needell and J.~A. Tropp.
\newblock {CoSaMP}: iterative signal recovery from incomplete and inaccurate
  samples.
\newblock {\em Communications of the ACM}, 53(12):93--100, 2010.

\bibitem{noack:03cyl}
B.~R. Noack, K.~Afanasiev, M.~Morzynski, G.~Tadmor, and F.~Thiele.
\newblock A hierarchy of low-dimensional models for the transient and
  post-transient cylinder wake.
\newblock {\em Journal of Fluid Mechanics}, 497:335--363, 2003.

\bibitem{Nyquist:1928}
H.~Nyquist.
\newblock Certain topics in telegraph transmission theory.
\newblock {\em Transactions of the A. I. E. E.}, pages 617--644, {FEB} 1928.

\bibitem{Patel:2013}
V.~M. Patel and R.~Chellappa.
\newblock {\em Sparse representations and compressive sensing for imaging and
  vision}.
\newblock Briefs in Electrical and Computer Engineering. Springer, 2013.

\bibitem{penlandMWR89}
C.~Penland.
\newblock Random forcing and forecasting using {P}rincipal {O}scillation
  {P}attern analysis.
\newblock {\em Mon.\ Weather Rev.}, 117(10):2165--2185, October 1989.

\bibitem{penlandJClimate93}
C.~Penland and T.~Magorian.
\newblock Prediction of {N}i\~{n}o 3 sea-surface temperatures using linear
  inverse modeling.
\newblock {\em J.\ Climate}, 6(6):1067--1076, June 1993.

\bibitem{Qi:2012}
H.~Qi and S.~M. Hughes.
\newblock Invariance of principal components under low-dimensional random
  projection of the data.
\newblock IEEE International Conference on Image Processing, October 2012.

\bibitem{Rowley:2009}
C.~W. Rowley, I.~Mezi\'c, S.~Bagheri, P.~Schlatter, and D.~S. Henningson.
\newblock Spectral analysis of nonlinear flows.
\newblock {\em Journal of Fluid Mechanics}, 641:115--127, 2009.

\bibitem{Sankaranarayanan:2012}
A.~C. Sankaranarayanan, P.~K. Turaga, R.~G. Baraniuk, and R.~Chellappa.
\newblock Compressive acquisition of dynamic scences.
\newblock In {\em Computer Vision--ECCV}, pages 129--142, 2010.

\bibitem{Schaeffer:2013}
H.~Schaeffer, R.~Caflisch, C.~D. Hauck, and S.~Osher.
\newblock Sparse dynamics for partial differential equations.
\newblock {\em Proceedings of the National Academy of Sciences USA},
  110(17):6634--6639, 2013.

\bibitem{schmid:2010}
P.~J. Schmid.
\newblock Dynamic mode decomposition of numerical and experimental data.
\newblock {\em Journal of Fluid Mechanics}, 656:5--28, August 2010.

\bibitem{Schmid:2011}
P.~J. Schmid.
\newblock Application of the dynamic mode decomposition to experimental data.
\newblock {\em Experiments in Fluids}, 50:1123--1130, 2011.

\bibitem{Schmid:2012}
P.~J. Schmid, D.~Violato, and F.~Scarano.
\newblock Decomposition of time-resolved tomographic {PIV}.
\newblock {\em Experiments in Fluids}, 52:1567--1579, 2012.

\bibitem{Shannon:1948}
C.~E. Shannon.
\newblock A mathematical theory of communication.
\newblock {\em Bell System Technical Journal}, 27(3):379--423, 1948.

\bibitem{Shi:2014}
J.~V. Shi, W.~Yin, A.~C. Sankaranarayanan, and R.~G. Baraniuk.
\newblock Video compressive sensing for dynamic {MRI}.
\newblock {\em submitted for publication}, 2013.

\bibitem{Sirovich:1987}
L.~Sirovich.
\newblock Turbulence and the dynamics of coherent structures, parts {I-III}.
\newblock {\em Q. Appl. Math.}, XLV(3):561--590, 1987.

\bibitem{Solomon:1988}
T.~H. Solomon and J.~P. Gollub.
\newblock Chaotic particle transport in time-dependent {Rayleigh-B}\'{e}nard
  convection.
\newblock {\em Physical Review A}, 38(12):6280--6286, 1988.

\bibitem{Tropp:2004}
J.~A. Tropp.
\newblock Greed is good: Algorithmic results for sparse approximation.
\newblock {\em IEEE Transactions on Information Theory}, 50(10):2231--2242,
  2004.

\bibitem{Tropp:2010}
J.~A. Tropp, J.~N. Laska, M.~F. Duarte, J.~K. Romberg, and R.~G. Baraniuk.
\newblock Beyond {N}yquist: Efficient sampling of sparse bandlimited signals.
\newblock {\em IEEE Transactions on Information Theory}, 56(1):520--544, 2010.

\bibitem{Tu:2014a}
J.~H. Tu, D.~M. Luchtenburg, C.~W. Rowley, S.~L. Brunton, and J.~N. Kutz.
\newblock On dynamic mode decomposition: theory and applications.
\newblock {\em submitted for publication}, 2013.

\bibitem{Tu:2014b}
J.~H. Tu, C.~W. Rowley, and J.~N. Kutz.
\newblock Spectral analysis of fluid flows using sub-{N}yquist rate {PIV} data.
\newblock {\em submitted for publication}, 2013.

\end{thebibliography}
}
\end{spacing}
\end{document}